\documentclass[a4paper,11pt]{article}

\usepackage[utf8]{inputenc}
\usepackage{amsmath}
\usepackage{amsfonts}
\usepackage{amssymb}
\usepackage{amsthm}
\usepackage[english]{babel}
\usepackage{esint}
\usepackage{hyperref}
\usepackage{fontenc}
\usepackage{fullpage}
\usepackage{graphicx}
\usepackage{subfigure}
\usepackage{authblk}
\usepackage{mathtools}
\usepackage{color}
\mathtoolsset{showonlyrefs}

\newtheorem{prop}{Proposition}
\newcommand{\ls}{\leqslant}
\newcommand{\gs}{\geqslant}
\newcommand{\Id}{\mathrm{Id}}
\newcommand{\tr}{\mathrm{tr}\,}
\newcommand{\BVdo}{\mathrm{BV}_{\diamond,1}}
\newcommand{\TV}{\mathrm{TV}}
\newcommand{\BV}{\mathrm{BV}}
\newcommand{\abs}[1]{\left \lvert#1 \right\rvert}
\newcommand{\R}{\mathbb R}
\DeclareRobustCommand{\rchi}{{\mathpalette\irchi\relax}}
\newcommand{\irchi}[2]{\raisebox{\depth}{$#1\chi$}}
\newcommand{\dd}{\, \mathrm{d}}
\renewcommand{\d}{\mathrm{d}}
\newcommand{\eps}{\varepsilon}
\renewcommand{\div}{\operatorname{div}}
\newcommand{\per}{\operatorname{Per}}
\newcommand{\set}[1]{\left \{ #1 \right\} }
\newcommand\restr[2]{{
\left.\kern-\nulldelimiterspace #1 \vphantom{\big|} \right|_{#2}}}

\newtheorem{thm}{Theorem}
\newtheorem{lem}{Lemma}
\newtheorem*{claim}{Claim}
\newtheorem{definition}{Definition}
\newtheorem{problem}{Problem}
\theoremstyle{remark}
\newtheorem{rem}{Remark}

\date{}
\title{Critical yield numbers and limiting yield surfaces of particle arrays settling in a Bingham fluid}
\author[1]{Jos\'{e} A. Iglesias \thanks{jose.iglesias@ricam.oeaw.ac.at}}
\author[1]{Gwenael Mercier \thanks{gwenael.mercier@ricam.oeaw.ac.at}}
\author[1,2]{Otmar Scherzer \thanks{otmar.scherzer@univie.ac.at}}
\affil[1]{\small Johann Radon Institute for Computational and Applied Mathematics (RICAM) \protect\\ Austrian Academy of Sciences, Linz, Austria.}
\affil[2]{\small Computational Science Center, University of Vienna, Vienna, Austria.}
\begin{document}
\maketitle

\begin{abstract}
We consider the flow of multiple particles in a Bingham fluid in an anti-plane shear flow configuration. The limiting situation in which the internal and applied forces balance and the fluid and particles stop flowing, that is, when the flow settles, is formulated as finding the optimal ratio between the total variation functional and a linear functional. The minimal value for this quotient is referred to as the critical yield number or, in analogy to Rayleigh quotients, generalized eigenvalue. This minimum value can in general only be attained by discontinuous, hence not physical, velocities. However, we prove that these generalized eigenfunctions, whose jumps we refer to as limiting yield surfaces, appear as rescaled limits of the physical velocities. Then, we show the existence of geometrically simple minimizers. Furthermore, a numerical method for the minimization is then considered. It is based on a nonlinear finite difference discretization, whose consistency is proven, and a standard primal-dual descent scheme. Finally, numerical examples show a variety of geometric solutions exhibiting the properties discussed in the theoretical sections.
\end{abstract}
\section{Introduction}
In this article, we investigate the stationary flow of particles in a Bingham fluid. Such fluids are important examples of 
non-Newtonian fluids, describing for instance cement, toothpaste, and crude oil \cite{VinWacAga05}. 
They are characterized by two numerical quantities: a yield stress $\tau_Y$ that must be exceeded for strain to appear, 
and a fluid viscosity $\mu_f$ that describes its linear behaviour once it starts to flow (see figure \ref{fig:ystress}).

\begin{figure}[htbp]
\centering
 \includegraphics[width = 0.5\textwidth]{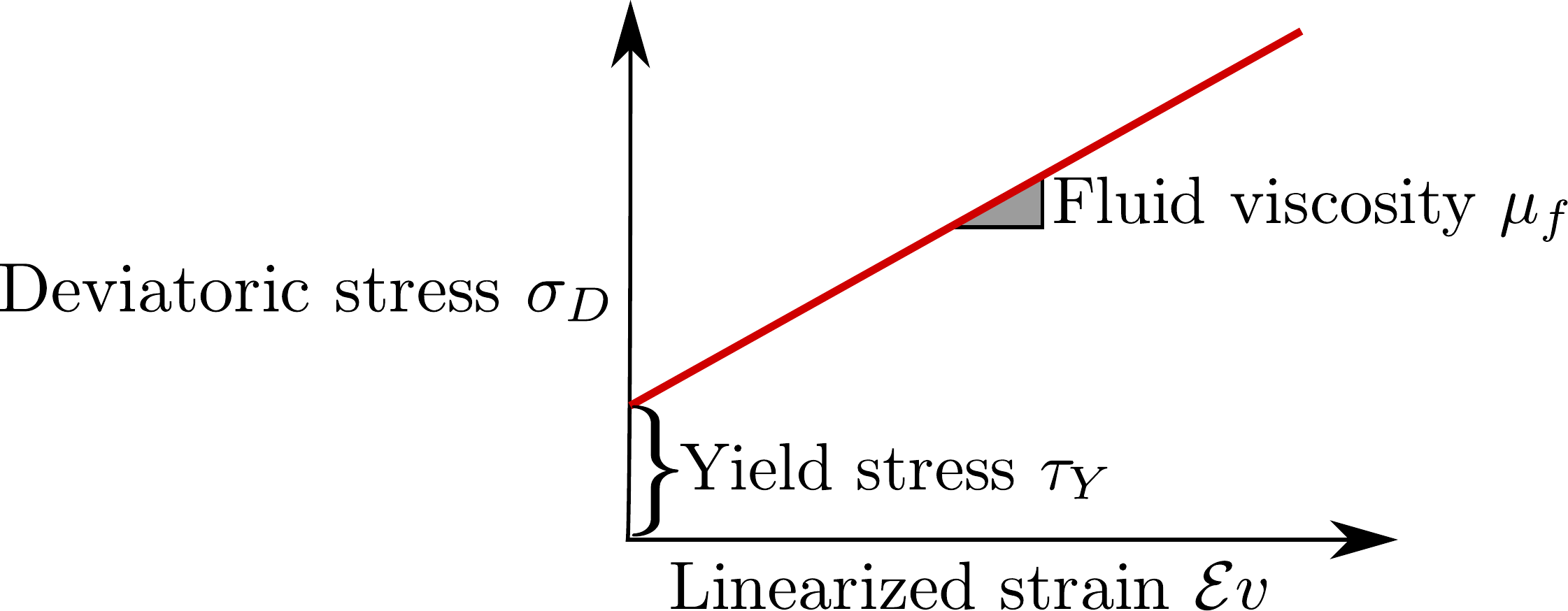}
 \caption{Relation between stress and strain in a Bingham fluid}
 \label{fig:ystress}
\end{figure}

An important property of Bingham fluid flows is the occurrence of plugs, which are regions where the fluid moves like a rigid body. Such rigid movements occur at positions where the stress does not exceed the yield stress. 

In this paper we consider anti-plane shear flow in an infinite cylinder, where an ensemble of inclusions move under their own weight inside a Bingham fluid of lower density, and in which the gravity and viscous forces are in equilibrium (cf \eqref{eq:balance}), therefore inducing a flow which is \emph{steady} or \emph{stationary}, that is, in which the velocity does not depend on time. For such a configuration, we are interested in determining the ratio between applied forces and the yield stress such that the Bingham fluid stops flowing completely. This ratio is called \emph{critical yield number}.

\paragraph*{Related work.}\mbox{}\\
To our knowledge, the first mathematical studies of critical yield numbers were conducted by
Mosolov \& Miasnikov \cite{MosMia65,MosMia66}, who also considered the anti-plane situation for flows inside a pipe. In particular, they discovered the geometrical nature of the problem and related the critical yield number to what in modern teminology is known as the Cheeger constant of the cross-section of the region containing the fluid. Very similar situations appear in the modelling of the onset of landslides \cite{Hild2002,Ionescu2005,Hassani2005}, where non-homogeneous coefficients and different boundary conditions arise. Two-fluid anti-plane shear flows that arise in oilfield cementing are studied in \cite{FriSch98,FriSch00}. Settling of particles under gravity, not necessarily in anti-plane configurations is also considered in \cite{JosMag01,PutFri10}. Finally, the previous work \cite{FriIglMerPoeSch17} also focuses in the anti-plane settling problem. There, the analysis is limited to the case in which all particles move with the same velocity and where the main interest is to extract the critical yield numbers from geometric quantities. In the current work we lift this restriction and focus on the calculations of the limiting velocities, also from a numerical point of view. Various applications of the critical yield stress of suspensions are pointed out in \cite[Section 4.3]{BalFriOva14}. On the numerical aspects, there are several methods available in the literature for the computation of limit loads \cite{CarComIonPey11} and Cheeger sets \cite{CarComPey09, CasFacMei09, BogBucFra18}, and both of these problems are closely related to ours, as we shall see below.

\paragraph*{Structure of the paper.}\mbox{}\\
We begin in Section \ref{sec:model} by recalling the mathematical models describing the stationary Bingham fluid flow in an anti-plane configuration, and an optimization formulation for determining the critical yield number. \\
Next, in Section \ref{sec:relax} we consider a relaxed formulation of this optimization problem, which is naturally set in spaces of functions of bounded variation, and show that the limiting velocity profile as the flow stops is a minimizer of this relaxed problem.\\
In Section \ref{sec:geomsols}, as in the case of a single particle \cite{FriIglMerPoeSch17}, we prove that there exists a minimizer that attains only two non zero velocity values.\\
Finally, in Section \ref{sec:numerics} we present a numerical approach to compute minimizers. This approach is based on the non-smooth convex optimization scheme of Chambolle-Pock \cite{ChaPoc11} and an upwind finite difference discretization \cite{ChaLevLuc11}. We prove the convergence of the discrete minimizers to continuous ones as the grid size decreases to zero. We then use this scheme to illustrate the theoretical results of Section \ref{sec:geomsols}.

\section{The model}\label{sec:model}
The constitutive law for an incompressible Bingham fluid in three dimensions is given by the von Mises criterion
\begin{equation}\label{eq:vonmises3D}
\left\{
\begin{aligned}
   \sigma_D &= \left(  \mu_f + \frac{ \tau_Y}{|\mathcal E v|} \right) \mathcal E v \quad &\text{if } & |\sigma_D| \gs \tau_Y, \\
  \mathcal E v &= 0 \quad &\text{if } &|\sigma_D| \ls \tau_Y,
\end{aligned}
\right.
\end{equation}
where $v$ is its velocity (for which incompressibility implies $\div v = 0$), and $\mathcal E v = (\nabla v + \nabla v^\top)/2$ is the linearized strain, $\nabla v \in \R^{3 \times 3}$ being the Jacobian matrix of the vector $v$. We denote by $\sigma_D$ the deviatoric part of the Cauchy stress tensor $\sigma(x,y,z) \in \R^{3 \times 3}_{\text{sym}}$, that is 
\begin{equation}\label{eq:devstr}\sigma = \sigma_D - p\,\Id, \end{equation} 
where $p$ is the pressure and $\tr \sigma_D = 0$. These equations state that as long as a certain stress is not reached, there is no response of the fluid (see Figure \ref{fig:ystress}).

The geometry we consider consists of a Bingham fluid filling a vertical cylindrical domain $\hat\Omega \times \R \subset \R^3$ and a solid inclusion $\hat\Omega_s  \times \R \subset  \hat\Omega \times \R$, where \[\hat \Omega_s = \bigcup_{i=1}^N \hat \Omega_s^i\] with $\hat\Omega_s^i \cap \hat\Omega_s^j = \emptyset$ and $\partial \hat\Omega \cap \partial \hat\Omega_s^i=\emptyset$, so that $\hat \Omega_s$ is composed of disconnected particles that do not touch the boundary of the domain. 
We denote by $\hat\Omega_f =  \hat\Omega \setminus  \hat\Omega_s$ the portion of the domain occupied by the fluid, and by $\rho_s, \rho_f$ the corresponding constant densities. We focus on a vertical stationary flow, meaning that the velocity is of the form $v=\hat\omega (0,0,1)^T$ and constant in time. Moreover, all quantities are invariant along the vertical direction, so we can directly consider a scalar velocity $\hat\omega:\hat\Omega \to \R$ ($\hat\omega$ is the velocity of the fluid on $\hat\Omega_f$ and of the solid in $\hat\Omega_s$), see Figure \ref{fig:antiplane}. For the rest of the article, the differential operators denoted by $\nabla$ and $\div$ are the two-dimensional ones.

Additionally to incompressibility, we consider the stronger condition of an \emph{exchange flow problem}, meaning that we require that the total flux across the horizontal slice is zero,
\begin{equation} \label{eq:exchange} 
\int_{\hat\Omega} \hat\omega = 0.
\end{equation}
 
A word on this condition is required. If the cylindrical domain was closed by a bottom fluid reservoir on which no-slip boundaries are assumed, one could use incompressibility, the divergence theorem and the boundary conditions to obtain \eqref{eq:exchange} in any horizontal plane. In our case, while not strictly consistent with an infinite cylinder, it is added as a modelling assumption, reflecting that the region of interest is far away from the bottom of the 3D domain. The same approximation has been used in previous works treating models of drilling and cementing of oil wells \cite{Fri98, FriSch98} and justified experimentally in \cite{HupHal07} with applications to magma in volcanic conduits.

In the anti-plane case, the Bingham constitutive law \eqref{eq:vonmises3D} can be written in terms of the vector of shear stresses $\hat \tau = (\sigma_{xz}, \sigma_{yz})$ to obtain
\begin{equation}\label{eq:vonmises2D}
\left\{
\begin{aligned}
   \hat\tau &= \left(  \mu_f + \frac{ \tau_Y}{|\nabla  \hat\omega|} \right) \nabla  \hat\omega \quad &\text{if } & \hat\tau \gs  \tau_Y, \\
  \nabla  \hat\omega &= 0 \quad &\text{if } &\hat\tau \ls  \tau_Y.
\end{aligned}
\right.
\end{equation}
Since the material occupying the region $\hat\Omega_s$ is perfectly rigid, the corresponding constitutive law is
\begin{equation}\label{eq:solid}
\nabla\hat\omega = 0 \quad  \text{on } \hat\Omega_s.
\end{equation}

Noting the decomposition of the stress tensor \eqref{eq:devstr}, the balance laws for the fluid and the solid particles then write
\begin{equation}
\label{eq:balance}
\left\{
\begin{aligned} 
 \div \hat\tau &= p_z - \rho_f  g \qquad \text{on } \hat\Omega_f, \\ 
 \int_{\partial \hat\Omega_s^i} \hat\tau \cdot n_f &+ \rho_s g \,|\hat\Omega_s^i|-b_i = 0,\\  
\end{aligned}
\right.
\end{equation}
with $p_z$ the pressure gradient along the vertical direction. The second equation in \eqref{eq:balance} expresses that for a steady fall motion, the gravity and buoyancy forces should be in equilibrium with the shear forces exterted by the fluid on each particle \cite{Wei72}. The buoyancy forces $b_i$ on each solid particle should be understood as resulting from Archimedes' principle and originating outside the region of interest, being exerted by the bottom reservoir of fluid. This interpretation implies that these forces are proportional to the volume of the solids and the vertical difference of pressure, a fact that we obtain as a consequence of the exchange flow condition in \eqref{eq:multipliers}. In this equation, $n_f$ is the exterior unit normal to $\partial \hat\Omega_f$, which at $\partial \hat \Omega _f \cap \partial \hat \Omega _s$ is the interior unit normal to $\partial \hat \Omega _s$.

These equations are complemented by the following boundary conditions: we assume that on the boundaries of $\hat\Omega$, we have a no-slip boundary condition
\begin{equation}\label{eq:bc1} \hat\omega = 0 \quad  \text{on } \partial  \hat\Omega, \end{equation}
and similarly we assume that $\hat \omega$ is continuous across the interface $\partial \hat\Omega_s$,
\begin{equation}\label{eq:transmission} [\omega]_{\partial\hat\Omega_s} = 0.\end{equation}

\begin{figure}[htbp]
    \begin{center}
		\includegraphics[width=.6\textwidth]{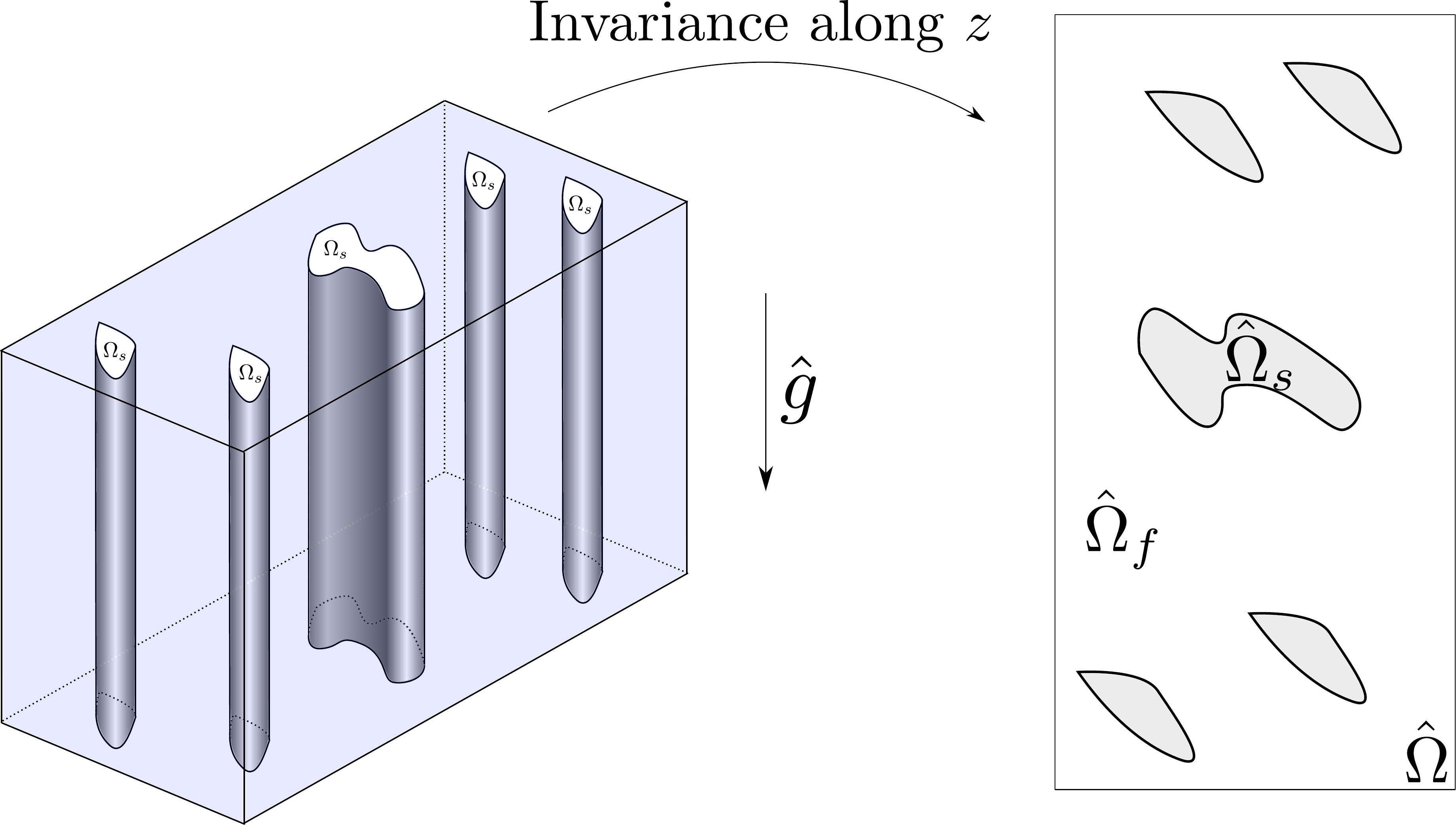}
    \end{center}
   \caption{Anti-plane situation: A falling cylinder, with gravity along its axis of symmetry.}
  \label{fig:antiplane}
\end{figure}

\subsection{Eigenvalue problems}
We assume that $\hat \Omega$ and $\hat \Omega_s$ are bounded and strongly Lipschitz, $\hat \Omega_s \subset \hat \Omega$, that $\partial \hat \Omega_s \cap \partial \hat \Omega = \emptyset$ and that $\hat \Omega_s$ has finitely many connected components. Following \cite{FriIglMerPoeSch17,PutFri10}, we introduce the functional 
\begin{equation}
\label{eq:Fhat}
\hat F(\hat\omega, m) := \left\{ \begin{aligned}
 &\frac{\mu_f}{2} \int_{\hat\Omega_f} |\nabla \hat\omega|^2 + \tau_Y \int_{\hat\Omega_f} |\nabla \hat\omega| - \rho_f \, g \int_{\hat \Omega_f} \hat\omega - \rho_s \, g \int_{\hat \Omega_s} \hat\omega + m \int_{\hat \Omega} \hat\omega &&\text{if } \hat\omega \in \hat H_\star\\
 &+\infty &&\text{else,}
\end{aligned}
\right.
\end{equation}
with the set of admissible velocities
\begin{equation}\label{eq:Hshat}
 \hat H_\star = \left\{ v \in H_0^1(\hat \Omega) \ \middle \vert \ \nabla v = 0 \text{ in } \hat\Omega_s \right\}.
\end{equation}
where the argument $m$ is a scalar multiplier for the exchange flow condition \eqref{eq:exchange}. Writing the Euler-Lagrange equations in the $\hat \omega$ argument at an optimal pair for the saddle point problem, we obtain a solution of our constitutive and balance equations \eqref{eq:vonmises2D} and \eqref{eq:balance}, with
\begin{equation}
\label{eq:multipliers}
p_z \equiv m \text{, and }b_i = p_z |\hat\Omega_s^i|. 
\end{equation}
Notice that since we work in $\hat H_\star$, the no-slip boundary condition \eqref{eq:bc1} and solid constitutive law \eqref{eq:solid} are automatically satisfied, and adequate testing directions are constant on connected components of $\hat \Omega_s$, which leads to the force balance condition in the second part of \eqref{eq:balance}. Condition \eqref{eq:transmission} is implied (in an appropriate weak form) by the fact that $\hat \omega \in H^1(\hat \Omega)$.

Since $\hat F$ is convex in its first argument and concave on the second, we can introduce the integral constraint in the space, and focus on the equivalent formulation of finding minimizers of
\begin{equation}
\label{eq:Ghat}
\hat G^\diamond(\hat\omega) := \left\{ \begin{aligned}
 &\frac{\mu_f}{2} \int_{\hat\Omega_f} |\nabla \hat\omega|^2 + \tau_Y \int_{\hat\Omega_f} |\nabla \hat\omega|- (\rho_s - \rho_f)\, g \int_{\hat \Omega_s} \hat\omega &&\text{if } \hat\omega \in \hat H_\diamond \\
 &+\infty &&\text{else,}
\end{aligned}
\right.
\end{equation}
over
\begin{equation}\label{eq:Hdhat}
 \hat H_\diamond = \left\{ v \in H_0^1(\hat \Omega) \ \middle \vert \ \int_{\hat\Omega} v =0, \ \nabla v = 0 \text{ in } \hat\Omega_s \right\}.
\end{equation}

We proceed to simplify the dimensions in the above functional, so that we can work with just one parameter. Assuming a given length scale $\hat L$, we define the buoyancy number $Y$ and a velocity scale $\hat \omega_0$ by
\begin{equation}\label{eq:scales}Y:=\frac{\tau_Y}{(\rho_s - \rho_f)g \hat L},\; \hat\omega_0:=\frac{(\rho_s - \rho_f)g\hat L^2}{\mu_f},\end{equation}
so that defining the rescaled velocity $\omega$ and corresponding domains by
\begin{equation}\label{eq:newvars}\omega(x):=\frac{\hat \omega(\hat L x)}{\hat \omega_0},\; \Omega := \frac{\hat \Omega}{\hat L},\; \Omega_f := \frac{\hat \Omega_f}{\hat L},\text{ and } \Omega_s := \frac{\hat \Omega_s}{\hat L},\end{equation}
we end up with the functional
\begin{equation}
\label{eq:GY}
G_Y^\diamond(\omega) := \left\{ \begin{aligned}
 &\frac{1}{2} \int_{\Omega_f} |\nabla \omega|^2 + Y \int_{\Omega_f} |\nabla \omega|-\int_{\Omega_s} \omega &&\text{if } \omega \in H_\diamond \\
 &+\infty &&\text{else,}
\end{aligned}
\right.
\end{equation}
to be minimized over 
\begin{equation}\label{eq:Hd}
H_\diamond = \left\{ v \in H_0^1( \Omega) \ \middle \vert \ \int_{\Omega} v =0, \ \nabla v = 0 \text{ in } \Omega_s \right\}.
\end{equation}
By the direct method it is easy to prove (see for instance \cite{FriIglMerPoeSch17}) that $G_Y^\diamond$ has a unique minimizer, which we denote by $\omega_Y$ and that corresponds to the weak solution of \eqref{eq:exchange}, \eqref{eq:vonmises2D}, \eqref{eq:solid}, \eqref{eq:balance}, \eqref{eq:bc1}, and \eqref{eq:transmission} in physical dimensions through the scaling in \eqref{eq:newvars}. Now, noticing that $u \mapsto Y\int_{\Omega_f} |\nabla u|-\int_{\Omega_s}u$ is convex, and that the G\^{a}teaux derivative of $u \mapsto \int_{\Omega_f} |\nabla u|^2$ at the point $\omega_Y$ in direction $h$ is $\int_{\Omega_f} \nabla \omega_Y \cdot \nabla h$, differentiating in the direction $v-\omega_Y$, as done in \cite[Section I.3.5.4]{DuvLio76} shows that for every $v \in H_\diamond$,
\begin{equation}
\int_{\Omega_f} \nabla \omega_Y \cdot  \nabla (v-\omega_Y) + Y \int_{\Omega_f} |\nabla v| - Y \int_{\Omega_f} |\nabla \omega_Y| \gs \int_{\Omega_s} (v - \omega_Y).
\label{eq:notEL}
\end{equation}

As in \cite{FriIglMerPoeSch17}, one can introduce 
\begin{equation}Y_c := \sup_{\omega \in H_\diamond} \frac{\int_{\Omega_s} \omega }{\int_{\Omega} | \nabla \omega|}\label{eq:eigenval} \end{equation} and test inequality \eqref{eq:notEL} with $v = 0$ and $v = 2\omega_Y$ to obtain
\begin{equation}\label{eq:zerograd1}
\int_\Omega |\nabla \omega_Y|^2 = \int_{\Omega_f} |\nabla \omega_Y|^2
= \int_{\Omega_s} \omega_Y - Y \int_{\Omega_f} |\nabla \omega_Y|.
\end{equation}
From this, and using the definition of $Y_c$ in \eqref{eq:eigenval} it follows that
\begin{equation}\label{eq:zerograd}
\int_\Omega |\nabla \omega_Y|^2 \ls   \int_{\Omega_f} |\nabla \omega_Y| \left[  \sup_{\omega \in H_\diamond}
\frac{\int_{\Omega_s} \omega }{\int_{\Omega} |\nabla \omega|} - Y \right]
= (Y_c-Y) \int_{\Omega_f} \abs{\nabla \omega_Y}.
\end{equation}
The last inequality implies, thanks to the homogeneous boundary conditions on $\omega$, that $\omega_Y = 0$ in $\Omega_f$ 
as soon as $Y \gs Y_c.$

\section{Relaxed problem and physical meaning}\label{sec:relax}
We determine the critical yield stress $Y_c$, defined in \eqref{eq:eigenval} and properties of the 
associated eigenfunction. The optimization problem \eqref{eq:eigenval} is equivalent to computing minimizers of the functional 
\begin{equation}
E(\omega):=\frac{\int_{\Omega} |\nabla \omega|}{\int_{\Omega_s} \omega} \text{ over } H_\diamond. 
\label{eq:buonum}
\end{equation}
Because $E$ might not attain a minimizer in $H_\diamond$, we consider a relaxed formulation on a subset 
of functions of bounded variation.

\subsection{Functions of bounded variations and their properties}
We recall the definition of the space of functions of bounded variation and some properties of such functions that we will use below. Proofs and further results can be found in \cite{AmbFusPal00}, for example.

\begin{definition}
Let $A \subset \R^2$ be open. A function $v \in L^1(A)$ is said to be of bounded variation
if its distributional gradient $\nabla v$ is a Radon measure with finite mass, which we denote by $\TV(v)$. In particular, if $\nabla v \in L^1(A)$, then $\TV(v)=\int_A \abs{\nabla v}$. Similarly, for a set $B$ with finite Lebesgue measure $|B| < +\infty$ we define its perimeter to be the total variation of its characteristic function $1_B$, that is, $\per(B)=\TV(1_{B})$.
\end{definition}

\begin{thm}
The space of functions of bounded variation on $A$, denoted $\BV(A)$, is a Banach space when associated with the norm
$$\Vert v \Vert_{\BV(A)} := \Vert v \Vert_{L^1(A)} + \TV(v)\;.$$
\end{thm}

The space of functions of bounded variation satisfies the following compactness property \cite[Theorem 3.44]{AmbFusPal00}:
\begin{thm}[Compactness and lower semi-continuity in $\BV$]\label{th:compact}
\label{thm:compact}
Let $v_n \in \BV(A)$ be a sequence of functions such that $\Vert v_n \Vert_{\BV(A)}$ is bounded. 
Then there exists $v \in \BV(A)$ for which, possibly upon taking a subsequence, we have 
$$v_n \xrightarrow{L^1} v.$$
In addition, for any sequence $(w_n)$ that converges to some $w$ in $L^1$,
$$\TV(w) \ls \liminf \TV(w_n).$$
\end{thm}
We frequently use the \emph{coarea} and \emph{layer cake formulas}:
\begin{lem} 
\label{lem:coarealayercake}
Let $u \in \BV(\R^2)$ with compact support, then the coarea formula \cite[Theorem 3.40]{AmbFusPal00}
\begin{equation}\label{eq:coarea}
\TV(u) = \int_{-\infty}^{\infty} \per(u > t) \dd t = \int_{-\infty}^\infty \per( u < t) \dd t
\end{equation}
holds. If $u \in L^1(\R^2)$ is non-negative, then we also have the layer cake formula \cite[Theorem 1.13]{LieLos01}
\begin{equation}  \label{eq:layercake} 
\int_{\R^2} u = \int_{0}^\infty |\{u > t \}| \dd t.
\end{equation}
\end{lem}
An important role in characterizing constrained minimizers of the $\TV$ functional is played by Cheeger sets, which we now define.
\begin{definition}
\label{de:cheeger}
 A set is called Cheeger set of $A \subseteq \R^2$ if it minimizes the ratio $\per(\cdot) / |\cdot|$ among the subsets of $A$.
\end{definition}
The following result is well known and has been stated for instance in \cite[Proposition 3.5, iii]{LeoPra16} and \cite[Proposition 3.1]{Par11}: 
\begin{thm}
\label{thm:cheeger}
For every non-empty measurable set $A \subseteq \R^2$ open, there exists at least one Cheeger set, and its characteristic function minimizes the quotient $u \mapsto \TV(u)/\|u\|_{L^1(A)}$ in $L^1(A)\setminus \{0\}$. Moreover, almost every level set of every minimizer of this quotient is a Cheeger set.
\end{thm}

\begin{rem}
Some sets may have more than one Cheeger set, which introduces nonuniqueness in the minimizers of the quotient $\TV(\cdot)/\|\cdot\|_{L^1(A)}$. One example is the set $\Omega$ of Figure \ref{fig:dumbbell} below.
\end{rem}

\subsection{Generalized minimizers of \texorpdfstring{$E$}{E}}
Using the compactness Theorem \ref{thm:compact}, it follows that the relaxed quotient 
\begin{equation*}
E(\omega):=\frac{\TV(\omega)}{\int_{\Omega_s} \omega}
\end{equation*}
of \eqref{eq:buonum} attains a minimizer in the space
$$ \mathcal{B}:=\left \{ v \in \BV(\R^2) \ \middle\vert \int_\Omega v = 0, \ \nabla v = 0 \text{ on } \Omega_s, \ v = 0 \text{ on } \R^2 \setminus \Omega \right\}.$$
Note that the quotient $E$ is invariant with respect to scalar multiplication, and we can therefore add the constraint 
\begin{equation}
\fint_{\Omega_s} v := \frac{1}{|\Omega_s|}\int_{\Omega_s} v= 1\label{eq:normalize}
\end{equation} 
to $\mathcal{B}$ without changing the minimal value of the functional $E$. Thus, the problem of minimizing $E$ over $\mathcal{B}$ is equivalent to the following problem:
\begin{problem}
\label{pr:1}
Find a minimizer of $\TV$ over the set
$$ \BV_\diamond := \left \{ v \in \BV(\R^2) \ \middle\vert \int_\Omega v = 0, \ \fint_{\Omega_s} v = 1, \ \nabla v = 0 \text{ on } \Omega_s, \ v = 0 \text{ on } \R^2 \setminus \Omega \right\}.$$
\end{problem}
By using standard compactness and lower semicontinuity results in $\BV(\R^2)$, it is easy to see \cite{FriIglMerPoeSch17} that there is at least one solution to Problem 1. In particular, we emphasize that all the constraints above are closed with respect to the $L^1$ topology.
\begin{rem}
Notice that $\BV_\diamond$ is larger than the optimization space \eqref{eq:bvdo} used in \cite{FriIglMerPoeSch17} , where it has been assumed that $v=\text{const.}$ in $\Omega_s$. See also Section \ref{sec:BV}.
\end{rem}

\subsection{The critical yield limit}
We investigate the limit of $\omega_Y$ (the minimizer of $G_Y^\diamond$, defined in \eqref{eq:GY}) 
when $Y \to Y_c$. For this purpose we first prove
\begin{prop}\label{prop:bounded}
 The quantity $\int_{\Omega_f} |\nabla \omega_Y|$ is nonincreasing with respect to $0 \ls Y \ls Y_C$. In particular, it is bounded.
 \end{prop}
\begin{proof}
 Let $Y_c \gs Y_1 > Y_2 \gs 0$. Then, from the definition \eqref{eq:GY} of $\omega_Y$ being a minimizer of $G_Y^\diamond$ it follows that
 $$\mathcal G^\diamond_{Y_2}(\omega_{Y_2}) \ls \mathcal G^\diamond_{Y_2}(\omega_{Y_1}) = \mathcal G^\diamond_{Y_1}(\omega_{Y_1}) + (Y_2-Y_1) \int_{\Omega_f} |\nabla \omega_{Y_1}|,$$
  $$\mathcal G^\diamond_{Y_1}(\omega_{Y_1}) \ls \mathcal G^\diamond_{Y_1}(\omega_{Y_2}) = \mathcal G^\diamond_{Y_2}(\omega_{Y_2}) + (Y_1-Y_2) \int_{\Omega_f} |\nabla \omega_{Y_2}|,$$
  and summing, we get
$$(Y_1 - Y_2) \left(\int_{\Omega_f} |\nabla \omega_{Y_2}| - \int_{\Omega_f} |\nabla \omega_{Y_1}|\right) \gs 0,$$
which implies the assertion.
\end{proof}

We are now ready to investigate the convergence of $\omega_Y$ and its rate.
\begin{thm}
 For $Y \nearrow Y_c$, we have 
 \begin{equation}\int_{\Omega} |\nabla \omega_Y|^2 \ls |\Omega_f| (Y_c - Y)^{2}.\label{eq:almostquadest} \end{equation}
 Moreover, the sequence of rescaled profiles 
\begin{equation}\label{eq:vy}
v_Y := \frac{\omega_Y}{\int_\Omega |\nabla  \omega_Y|}
\end{equation}
converges in the sense of Theorem \ref{thm:compact}, up to possibly taking a sequence, to a solution of Problem \ref{pr:1}.
\end{thm}
\begin{proof}
The first part of the proof is already presented in \cite[Section VI 8.3, Equation (8.20)]{DuvLio76} but we reproduce it here for convenience.
 As before, let $Y_c \gs Y_1 > Y_2 \gs 0$. We use \eqref{eq:notEL} for $Y_1$ and $v=\omega_{Y_2}$ as well as the same inequality for $Y_2$ and $v = \omega_{Y_1}$ and sum the inequalities obtained to get
 $$\int_{\Omega_f} |\nabla \omega_{Y_1} - \nabla \omega_{Y_2}|^2 \ls (Y_1 - Y_2) \left( \int_{\Omega_f} |\nabla \omega_{Y_2}| - |\nabla \omega_{Y_1}| \right).$$
 With $Y_1 = Y_c$ and $Y_2$ a generic $Y$, and since $\omega_{Y_c}=0$, the above implies
 \begin{equation} \int_{\Omega_f} |\nabla \omega_Y|^2 \ls (Y_c-Y) \int_{\Omega_f} |\nabla \omega_Y|. \label{eq:rate12} \end{equation}
 On the other hand, the Cauchy-Schwarz inequality gives
$$ \int_{\Omega_f} |\nabla \omega_Y| \ls |\Omega_f|^{1/2} \left( \int_{\Omega_f} |\nabla \omega_Y|^2 \right)^{1/2}.$$
 Putting these two inequalities together, we obtain
 $$  \int_{\Omega_f} |\nabla \omega_Y|^2 \ls |\Omega_f|^{1/2} (Y_c-Y)  \left( \int_{\Omega_f} |\nabla \omega_Y|^2 \right)^{1/2}$$
 which leads to \eqref{eq:almostquadest}.

Now, the associated functions $v_Y$, defined in \eqref{eq:vy}, have total variation $1$ and zero mean. From Theorem \ref{thm:compact} it follows that $v_Y$ converges in $L^1$ to some $v_c$. Now, it follows directly from \eqref{eq:rate12} and \eqref{eq:zerograd1} that 
\begin{equation}\lim_{Y \to Y_c} \frac{Y \int_{\Omega} |\nabla  \omega_Y|}{\int_{\Omega_s} \omega_Y} = 1,\label{eq:ratiolimit}\end{equation} 
and therefore, using the $L^1$ convergence of $v_Y$, its definition \eqref{eq:vy} and that $\int_{\Omega} |\nabla \omega_y|=\int_{\Omega_f} |\nabla \omega_y|$, \eqref{eq:ratiolimit} implies 
$$ \int_{\Omega_s} v_c = \lim_{Y \to Y_c} \int_{\Omega_s} v_Y = \lim_{Y \to Y_c}  \frac{\int_{\Omega_s} \omega_Y}{\int_{\Omega_f} |\nabla \omega_y|} = Y_c.$$
Recalling that $\TV( v_Y) = 1$, the semi-continuity of the total variation with respect to $L^1$ convergence implies $\TV( v_c) \ls 1$, which yields
$$ Y_c \int_\Omega |\nabla  v_c| - \int_{\Omega_s} v_c  \ls 0,$$ 
which can be rewritten as
$$Y_c \ls \frac{\int_{\Omega_s} v_c}{\int_\Omega |\nabla  v_c|}$$
so $v_c$ is a maximizer of $v \mapsto \frac{\int_{\Omega_s} v}{\int_\Omega |\nabla v|} $.
\end{proof}
From the above result, we see that a minimizer of the quotient $\frac{\int_\Omega |\nabla  v|}{\int_{\Omega_s} v}$ can be obtained as a limit of rescaled physical velocities, and therefore carries information about their geometry. For this reason, we will focus on these minimizers in the following.

\section{Piecewise constant minimizers}\label{sec:geomsols}
We prove the existence of solutions of Problem \ref{pr:1} with particular properties. In our previous work \cite{FriIglMerPoeSch17} this problem was considered under the assumption that the velocity is constant in the whole $\Omega_s$. In the situation considered here, the physical velocity $\omega$ is constant only on every connected component of $\Omega_s$, and the velocity of \emph{each} solid particle is an unknown. Therefore, the candidates of limiting profiles $v$ over which we optimize (belonging to $\BV_\diamond$) also satisfy $\nabla v = 0$ on $\Omega_s$.

\subsection{A minimizer with three values}\label{sec:3level}
\begin{thm}\label{thm:3val}There is a solution of Problem \ref{pr:1} that attains only two non-zero values.
\end{thm} 

The same result has been proved in \cite{FriIglMerPoeSch17} in the simpler situation when the velocities were considered uniformly 
constant on the whole $\Omega_s$. For the proof of Theorem \ref{thm:3val}, we proceed in two steps:
\begin{enumerate}
\item We prove the existence of a minimizer for Problem \ref{pr:1} which attains only finitely many values. This is accomplished by convexity arguments reminiscent of slicing by the coarea \eqref{eq:coarea} and layer cake \eqref{eq:layercake} formulas, but more involved.
\item When considered over functions with finitely many values, the minimization of the total variation with integral constraints is a simple finite-dimensional optimization problem, and standard linear programming arguments provide the result.
\end{enumerate}

The core of the proof of Theorem \ref{thm:3val} is the following lemma, that states that a simplified version of the minimization problem can be solved with finitely many values.
\begin{lem}
Let $\Omega_1 \subset \Omega_0$ be two bounded measurable sets, $\nu \in \R$. Then, there exists a minimizer of $\TV$ on the set 
\begin{equation}
\mathcal{A}_\nu(\Omega_0, \Omega_1):= \set{ v \in \BV(\R^2) \ \middle | \ \restr{v}{\R^2 \setminus \Omega_0} \equiv 0, \ \restr{v}{\Omega_1} \equiv 1 , \ \int_{\R^2} v = \nu}, 
\label{eq:adm3} 
\end{equation}
where the range consists of at most five values, one of them being zero.
\label{lem:5val}
\end{lem}
In turn our proof of Lemma \ref{lem:5val} is based on the following minimizing property of level sets, which we believe could be of interest in itself.
\begin{lem}
Let $\Omega_0, \Omega_1, \nu$ and $\mathcal{A}_\nu(\Omega_0, \Omega_1)$ be as in Lemma \ref{lem:5val}, and $u$ a minimizer of $\TV$ in $\mathcal{A}_\nu(\Omega_0, \Omega_1)$. Assume further that $u$ has values only in $[0,1]$, and denote $E_s := \{u > s\}$. Let $s_0$ be a Lebesgue point of $ s \mapsto \per(E_s)$ and $s \mapsto |E_s|$ (these two functions are measurable, so almost every $s \in [0,1]$ is a Lebesgue point for them). Then $1_{E_{s_0}}$ minimizes $\TV$ in $\mathcal{A}_{|E_{s_0}|}(\Omega_0, \Omega_1)$.
 \label{lem:minTV}
\end{lem}
The proofs of these two lemmas are located after the proof of Theorem \ref{thm:3val}.
\begin{proof}[Proof of Theorem \ref{thm:3val}]
\textbf{Step 1. A minimizer with finite range.}

To begin the proof, we assume that we are given a minimizer $u$ of the total variation in $\BV_\diamond$, that is, a solution of Problem \ref{pr:1}. We represent $\Omega_s$ by its connected components $\Omega^i_s$, $i=1,\ldots,N$,
\begin{equation}\Omega_s = \bigcup_{i=1}^N \Omega^i_s.\label{eq:domdec}\end{equation}
Since $u$ belongs to $\BV_\diamond$, $u$ is constant on every $\Omega^i_s$, and we introduce the constants $\gamma_i$ such that
\begin{equation}
\restr{u}{\Omega^i_s}=\gamma_i.\label{eq:vi}\end{equation}
We can assume that $\gamma_i \ls \gamma_{i+1}$. Note that the constraint \eqref{eq:normalize} reads
\begin{equation}
\frac{1}{\sum_{i=1}^n |\Omega^i_s|} \sum_{i=1}^N \gamma_i |\Omega^i_s| = 1.
\label{eq:discconstr}
\end{equation}

Defining
$$u_i := u \cdot 1_{\{\gamma_i < u < \gamma_{i+1}\}} + \gamma_i 1_{\{u \ls \gamma_i\}} + \gamma_{i+1} 1_{\{u \gs \gamma_{i+1}\}},$$
we have
$$u = \sum_{i=1}^N \big( u_i - \gamma_i \big).$$
Notice that each $u_i$ minimizes the total variation among functions with fixed integral $\int_\Omega u_i$, and satisfying the boundary conditions $u = \gamma_i$ on $\{u \ls \gamma_i\}$ and $u=\gamma_{i+1}$ on $\{u \gs \gamma_{i+1}\}$. 

As a result, the function $v_i := \frac{u_i - \gamma_i}{\gamma_{i+1} - \gamma_i} $ minimizes the total variation with constraints $\restr{v_i}{\R^2 \setminus \{u > \gamma_i\}} \equiv 0$,  $\restr{v_i}{\{u \gs \gamma_{i+1}\}} \equiv 1$ and prescribed integral. Lemma \ref{lem:5val} (applied with $\Omega_0 =  \{u > \gamma_i\}$ and $\Omega_1 = \{u \gs \gamma_{i+1} \}$)
shows that $v_i$ can be replaced by a five level-set function $\tilde v_i$ which has total variation smaller or equal to $\TV(v_i)$. Hence $u_i$ can be replaced by the five level-set function $\tilde u_i := \gamma_i + \tilde v_i (\gamma_{i+1} - \gamma_i)$ without increasing the total variation.

Therefore, the finitely-valued function
$$ \tilde u :=  \sum_{i=1}^N \big( \tilde u_i - \gamma_i \big)$$ is again a solution of Problem \ref{pr:1} (the functions $u$ and $\tilde u$ coincide on $\Omega_s$, so the constraint $\fint_{\Omega_s} \tilde u =1$ is satisfied).

\textbf{Step 2. Construction of a three-valued minimizer.}

Step 1 provides a solution $\tilde u$ of Problem \ref{pr:1} that reaches a finite number (denoted as $p+1$) of values. We denote its range (listed in increasing order) by 
$$\{ \gamma_{p^-} , \cdots, \gamma_{-1}, 0, \gamma_1, \cdots,\gamma_{p^+} \}$$
where $p^- \ls 0 \ls p^+$, $p^+ - p^- = p$  and $\gamma_i < 0$ for $i< 0$ and $\gamma_i > 0$ for $i >0.$

Let us now define, for $i <0$, $E_i := \{\tilde u \ls \gamma_i\}$ and $\alpha_i := \gamma_i - \gamma_{i+1}$ and for $i > 0$, $E_i := \{ \tilde u \gs \gamma_i\}$ and $\alpha_i := \gamma_i - \gamma_{i-1}.$
The function $\tilde u$ then writes
\begin{equation} \tilde u = \sum_{\substack{ i = p^- \\ \mkern-15mu i \neq 0}}^{p^+} \alpha_i 1_{E_i}\label{eq:pwansatz}\end{equation} where $E_i \subset E_j$ whenever $i<j<0$ or $i>j>0.$

We also have
\begin{align}
\TV(\tilde u) &= \sum_{\substack{ i = p^- \\ \mkern-15mu i \neq 0}}^{p^+} |\alpha_i| \per(E_i),\label{eq:tvdisc} \\ 
\int_{\Omega} \tilde u &=\sum_{\substack{ i = p^- \\ \mkern-15mu i \neq 0}}^{p^+} \alpha_i |E_i|, \label{eq:intdisc} \\ 
\int_{\Omega_s} \tilde u &= \sum_{\substack{ i = p^- \\ \mkern-15mu i \neq 0}}^{p^+} \alpha_i |E_i^s| \label{eq:intsdisc}
\end{align}

where $E_i^s = E_i \cap \Omega_s$.

Since $\tilde u$ is a solution to Problem \ref{pr:1}, the collection $(\alpha_i)$ minimizes $\sum_i |\alpha_i| \per(E_i) $ with constraints
$$\sum_{\substack{ i = p^- \\ \mkern-15mu i \neq 0}}^{p^+} \alpha_i |E_i| = 0 \quad \text{and} \quad \sum_{\substack{ i = p^- \\ \mkern-15mu i \neq 0}}^{p^+} \alpha_i |E_i^s| = |\Omega_s|$$
as well as $\alpha_i < 0$ for $i< 0$ and $\alpha_i >0$ for $i>0.$ The constraint on the sign of the $\alpha_i$ is made such that the formula \eqref{eq:tvdisc} holds. Indeed, if the $\alpha_i$ change signs, the right hand side of \eqref{eq:tvdisc} is only an upper bound for $\TV(\tilde u)$.

Introducing the vectors
\begin{equation*}
 \begin{aligned}
  a &= (\per(E_{p^-}), \cdots, \per(E_{p^+})),\\
  b &= (|E_{p^-}^s|,\cdots,|E_{p^+}^s|),\\
  c &= (|E_{p^-}|,\cdots,|E_{p^+}|),\\
  x &= (\alpha_{p^-}, \cdots,\alpha_{p^+})\,,
 \end{aligned}
\end{equation*}
minimizing \eqref{eq:tvdisc} for $\tilde u$ of the form \eqref{eq:pwansatz} and with the constrained mentioned above is reformulated into finding a minimizer of 
\begin{equation*}
 \begin{aligned}
  &(a,x) \to \left \vert a^T |x| \right \vert_{\ell^1}\,,\\
  x &\text{ s.t. }b^T x =|\Omega_s| \text{ and } c^T x = 0\;.
 \end{aligned}
\end{equation*}
Denoting by $\sigma \in \{-1,1\}^{p} \subseteq \R^{p}$ indexed by $i \in \{p^-, \cdots,p^+\}$ with $\sigma_i = -1$ for $i< 0 $ and $\sigma_i = 1$ for $i>0$,  this minimization problem can be rewritten as
\begin{equation} \label{eq:linprog} \min_{ x \in \R^{p+1}} 
\left \{ a^T (\sigma:x) = (\sigma : a)^T x \ \middle \vert \ b^T x = |\Omega_s|,\ c^T x = 0,\ \sigma:x \gs 0 \right\},\end{equation}
where $\sigma:x := (x_1 \sigma_1, \cdots, x_p \sigma_p, x_{p+1} \sigma_{p+1})$. The space of constraints is then a (possibly empty) polyhedron given by the intersection 
of the quadrant $\sigma:x \gs 0$ with the two hyperplanes $c^T x = 0$ and $b^T x =|\Omega_s|$. Now for a point of a polyhedron in $\R^{p}$ to be a vertex, we must have that at least $p$ constraints are active at it. Therefore, at least $p-2$ of these constraints should be of those defining the quadrant $\sigma : x \gs 0$, meaning that at a vertex, at least $p-2$ coefficients of $x$ are zero.

This polyhedron could be unbounded, but since $a \gs 0$ and 
$\sigma:x \gs 0$ componentwise, the minimization of $a^T (\sigma:x)$ must have at least one solution in it. 
Moreover, since it is contained in a quadrant ($\sigma : x \gs 0$), it clearly does not contain any line, so it must have at least one vertex (\cite[Theorem 2.6]{BerTsi97}). 
Since the function to minimize is linear in $x$, it has a minimum at one such vertex 
(\cite[Theorem 2.7]{BerTsi97}). That proves the existence of a 
minimizer of \eqref{eq:linprog}  with at least $p-2$ of the $(\alpha_i)$ being zero. This corresponds to a minimizer for Problem \ref{pr:1} which has only two level-sets with nonzero values, finishing the proof of Theorem \ref{thm:3val}.\end{proof}

\subsubsection{Proof of Lemma \ref{lem:5val}}
\begin{proof}[Proof of Lemma \ref{lem:5val}]
For conciseness, we denote the set $\mathcal{A}_\nu(\Omega_0, \Omega_1)$ by $\mathcal{A}$. Let $w$ be an arbitrary minimizer of $\TV$ in $\mathcal{A}$. Splitting $w$ at $0$ and $1$ we can write 
\begin{equation}\label{eq:wdecomp}w = (w^{1+}-1) + w^{(0,1)} - w^-\end{equation}
with $w^{1+} := w \cdot 1_{w \gs 1} + 1_{w < 1}$, $w^{(0,1)} = w \cdot 1_{0 \ls w \ls 1} + 1_{w > 1}$, and $w^-$ the usual negative part. We see from the coarea formula that 
\begin{align*}\TV(w) &= \int_{s \ls 0} \per(w \gs s) + \int_{0<s< 1} \per(w \gs s) + \int_{s \gs 1} \per(w \gs s) \\ 
&= \TV(w^-) + \TV(w^{(0,1)}) + \TV(w^{1+}).\end{align*}
With this splitting, $w^-$ can be seen to be a minimizer of $\TV$ over 
\begin{equation*}
 \mathcal{A}^- := 
 \set{v \in \BV(\R^2) \ \middle \vert \ v=0 \text{ on } \set{w > 0} \cup \R^2 \backslash \Omega_0, \ \int_\Omega v = \int_\Omega w^-}.
\end{equation*}

By Theorem \ref{thm:cheeger}, almost every level set of $w^-$ is a Cheeger set of 
$\Omega_0 \setminus \{w >0\}$, the complement of $\set{w > 0} \cup \R^2 \backslash \Omega_0$. In particular, if we replace $w^-$ by $\frac{\int_\Omega w^-}{|\mathcal C_0|}1_{\mathcal C_0}$, where $\mathcal C_0$ is one such Cheeger set, the total variation doesn't increase. Therefore, there exists a minimizer $\tilde w^-$ of $\TV$ on $\mathcal{A}^-$ that reaches only one non-zero value.

With an analogous argumentation we see that, because $w^{1+}$ minimizes $\TV$ on the set
\begin{equation*}
 \mathcal{A}^{1+} := \set{v \in \BV(\R^2) \ \middle \vert \ v = 1 \text{ on } \{w < 1\}, \ \int_{\Omega} v = \int_{\Omega} w^{1+}},
\end{equation*}
there exists a minimizer $\tilde w^{1+}$ that writes
$$\tilde w^{1+} = 1 + \zeta 1_{\mathcal C_1}$$
where $\mathcal C_1$ is a Cheeger set of $\{w \gs 1\}$ and $\zeta \gs 0$ is a constant.

Moreover, defining $$\mu := \int_\Omega w^{(0,1)},$$ $w^{(0,1)}$ minimizes $\TV$ on the set 
\begin{equation*}
 \mathcal{A}^{(0,1)}_\mu := \set{ v \in \BV(\R^2) \ \middle \vert \ v = 1 \text{ on } \{w \gs 1\}, v = 0 \text{ on } \{w \ls 0\} \text{ and } \int_\Omega v = \mu}.\label{eq:A01}
\end{equation*}
The remainder of the proof consists in showing that there exists a minimizer of $\TV$ in $\mathcal{A}^{(0,1)}_\mu$ that attains only three values. 
Since $w^{(0,1)}$ is one of them, there exists some minimizer of $\TV$ in $\mathcal{A}^{(0,1)}_\mu$ with values in $[0,1]$. We denote by $u$ a generic one.
In what follows, we denote by $E_s := \{u > s\}$ the level-sets of $u$.

Noticing that $\mathcal{A}^{(0,1)}_\mu=\mathcal{A}_\mu(\{w \ls 0\},\{w \gs 1\})$, we can use Lemma \ref{lem:minTV} to obtain that for almost every $s$, $1_{E_{s}}$ minimizes $\TV$ in $\mathcal{A}^{(0,1)}_{|E_{s}|}$. That implies in particular that for a.e. $s$, $E_s$ minimizes perimeter with fixed mass. We introduce $E_s^{(1)}$ the set of points of density $1$ for $E_s$ and $E_s^{(0)}$ the set of points of density 0 for $E_s$, that is
$$ E_s^{(1)} := \set{x \in \Omega \, \middle \vert \, \lim_{r \to 0} \frac{ |E_s \cap B_r(x)|}{|B_r(x)|} = 1} \qquad \text{and} \qquad E_s^{(0)} := \set{x \in \Omega \, \middle \vert \, \lim_{r \to 0} \frac{ |E_s \cap B_r(x)|}{|B_r(x)|} = 0}.$$
Lebesgue differentiation theorem implies that $E_s^{(1)} = E_s$ and $E_s^{(0)} = \Omega \setminus E_s$ a.e. 

Now, since the level-sets are nested, the function $s \mapsto |E_s|$ is nonincreasing. Therefore, there exists $s_\mu$ such that $$\text{for }s > s_\mu, \,|E_s| \ls \mu \text{, and for }s < s_\mu,\, |E_s| \gs \mu.$$
Let us now define
$$E^+ := \bigcup_{s > s_\mu} E^{(1)}_s \qquad \text{and} \qquad E^- := \bigcap_{s < s_\mu} \Omega \setminus E^{(0)}_s.$$
We then have the following fact, to be proved below:
\begin{claim}
If $E^\pm$ is not empty, $1_{E^\pm}$ minimizes total variation in $\mathcal{A}^{(0,1)}_{|E^\pm|}$, with $|E^+| \ls \mu \ls |E^-|.$
\end{claim}

To finish the proof of Lemma \ref{lem:5val}, we distinguish two alternatives. Either $E^+$ or $E^-$ has mass $\mu$, in which case the claim above implies Lemma \ref{lem:5val}, or $E^\pm$ are both nonempty and
$$ |E^+| < \mu \qquad \text{and} \qquad |E^-| > \mu.$$
In the second case, let $s< s_\mu$. Then, $|E^-| \in (|E^+|,|E_s|)$ and there exists $t = \frac{|E^-| - |E^+|}{|E_s|-|E^+|}$ such that $|E^-| = t |E_s| + (1-t) |E^+|.$ The function $t 1_{E_s} + (1-t) 1_{E^+}$ therefore belongs to $\mathcal{A}^{(0,1)}_{|E^-|}$. Since $1_{E^-}$ is a minimizer of $\TV$ in this set, one must have
$$ \per(E^-) \ls \TV(t 1_{E_s} + (1-t) 1_{E^+}) \ls  \frac{|E^-| - |E^+|}{|E_s|-|E^+|} \per(E_s) + \frac{|E_s| - |E^-|}{|E_s|-|E^+|} \per(E^+).$$
This equation rewrites
\begin{equation} \per(E_s) \gs \frac{|E_s| - |E^+|}{|E^-| - |E^+|} \per(E^-) + \frac{|E^-| - |E_s|}{|E^-| - |E^+|} \per(E^+). \label{eq:perEs}\end{equation}
Similarly, if $s > s_\mu$, one has $|E_s| < |E^+|$ and $|E^+|$ is a convex combination of $\{|E^-|, |E_s|\}.$ The same steps lead to the same \eqref{eq:perEs}.
Finally, one just write (we use \eqref{eq:perEs}, the coarea and the layer-cake formulas)
\begin{align*}
\TV(u) &= \int_0^1 \per(E_s) \gs \int_0^1 \frac{\left( |E_s| - |E^+| \right)\per(E^-) + \left(|E^-| - |E_s| \right)  \per(E^+) }{|E^-| - |E^+|} \\
& \gs \int_0^1 \frac{\per(E^-) - \per(E^+)}{|E^-| - |E^+|} |E_s| +  \frac{|E^-|  \per(E^+) - |E^+|\per(E^-) }{|E^-| - |E^+|} \\
& = \frac{\per(E^-) - \per(E^+)}{|E^-| - |E^+|} \mu +  \frac{|E^-|  \per(E^+) - |E^+|\per(E^-) }{|E^-| - |E^+|} \\
& = \TV\left(\lambda 1_{E^-} + (1-\lambda) 1_{E^+} \right)
\end{align*}
with $\lambda = \frac{\mu - |E^+|}{|E^-| - |E^+|}$. 

As a result, one can replace $w^{(0,1)}$ in the decomposition \eqref{eq:wdecomp} by a three valued minimizer $\tilde w^{(0,1)}$ of $\TV$ in $\mathcal{A}^{(0,1)}_\mu$. Therefore, combining the three modified parts we see that there exists a minimizer in $\mathcal{A}$
$$\tilde w := (\tilde w^{1+}-1) + \tilde w^{(0,1)} - \tilde w^- $$
which attains at most five values.
\end{proof}

\begin{proof}[Proof of claim]
By Lemma \ref{lem:minTV}, $1_{E_s^{(1)}}$ minimizes total variation in $\mathcal{A}^{(0,1)}_{|E_s|}$ for almost every $s$. Then, let us select a decreasing sequence $s_n \searrow s_\mu$ such that for each $n$, $1_{E_{s_n}^{(1)}}$ minimizes total variation in $\mathcal{A}^{(0,1)}_{|E_{s_n}|}$. Since $E_{s_n}^{(1)} \to E^+$ in $L^1$, one has $|E^+| = \lim |E_{s_n}^{(1)}| = \lim |E_{s_n}|$ and the semicontinuity for the perimeter gives
$$\per(E^+) \ls \liminf \per(E_{s_n}^{(1)}).$$
In fact, the sequence $\per(E_{s_n}^{(1)})$ is bounded. To see this, we fix a value $\hat s < s_\mu$ and since $E_{s_1} \subset E_{s_n}^{(1)} \subset E_{\hat s}$ we can write for some $t_n \in (0,1)$
$$|E_{s_n}^{(1)}|=t_n |E_{\hat s}| + (1-t_n) |E_{s_1}|.$$
Therefore, applying Lemma \ref{lem:minTV} again we obtain
$$\per(E_{s_n}^{(1)}) \ls \TV\left(t_n 1_{E_{\hat s}} + (1-t_n) 1_{E_{s_1}}\right) \ls \per(E_{\hat s})+\per(E_{s_1}).$$

Now, let us assume that there exists $v \in \BV(\Omega)$ with $\int v = |E^+|$ and $\TV(v) < \per(E^+) - \varepsilon$. By the above, for every $\delta >0$ we can find $n$ such that $ |E^+| \gs |E_{s_n}| \gs |E^+| - \delta$  and
$$\per(E^+) \ls \per(E_{s_n}^{(1)}) + \delta.$$
Now, if $\delta < \varepsilon/10$ is small enough, we can find a ball $B_n \subset \Omega$ such that $\int_\Omega v \cdot 1_{\Omega \setminus B_n} = |E_{s_n}|$ and $\Vert v \Vert_\infty \per(B_n) \ls \varepsilon/10$, so we get 
\begin{equation}
\begin{aligned}\TV(v \cdot 1_{\Omega \setminus B_n}) &\ls \TV(v) + \Vert v \Vert_\infty \per(B_n) \ls \per(E^+) -\varepsilon + \Vert v \Vert_\infty \per(B_n) \\
&\ls \per(E_{s_n}^{(1)}) + \Vert v \Vert_\infty \per(B_n) + \delta - \varepsilon \ls \per(E_{s_n}^{(1)}) - \frac{\varepsilon}{2},\end{aligned}\end{equation}
and therefore we get a contradiction with the $\TV$-minimality of $E_{s_n}^{(1)}$.

Selecting an increasing sequence $\tilde s_n \nearrow s_\mu$ and such that $\Omega \setminus E_{s_n}^{(0)}$ minimizes $\TV$ in $\mathcal{A}^{(0,1)}_{|E_{s_n}|}$, we obtain similarly that $1_{E^-}$ minimizes $\TV$ in $\mathcal A^{(0,1)}_{|E^-|}.$
\end{proof}
\subsubsection{Proof of Lemma \ref{lem:minTV}}
\begin{proof}[Proof of Lemma \ref{lem:minTV}]
Since the arguments $\Omega_0, \Omega_1$ are fixed for the course of this proof, we will denote the sets $\mathcal{A}_\tau(\Omega_0, \Omega_1)$ by $\mathcal{A}_\tau$ for each $\tau>0$. First, note that for every $s_1 < s_2$, the function
$$ u_{[s_1,s_2]} := s_2 1_{E_{s_2}} + u \cdot 1_{[s_1,s_2]} + s_1 1_{u<s_1}$$ 
is such that $ v:= \frac{u_{[s_1,s_2]} - s_1}{s_2 - s_1}$ minimizes the total variation in $\mathcal{A}_{\int v}$. Indeed, if $\hat v \in \mathcal{A}_{\int v}$ with $\TV(\hat v) < \TV(v)$, then $\TV(\hat v(s_2-s_1)+s_1) < \TV(u_{[s_1,s_2]})$. Since $u = (u \cdot 1_{u< s_1} -s_1) + u_{[s_1,s_2]} + (u \cdot  1_{u> s_2} - s_2)$, then we would have 
\begin{align*} \TV(u) &= \TV((u \cdot 1_{u< s_1} -s_1)) + \TV(u_{[s_1,s_2]}) + \TV((u \cdot  1_{u> s_2} - s_2)) \\
 & >  \TV((u \cdot 1_{u< s_1} -s_1)) + \TV(\hat v(s_2-s_1)+s_1) + \TV((u \cdot  1_{u> s_2} - s_2)) \\
 & \gs \TV\big((u \cdot 1_{u< s_1} -s_1) + (\hat v(s_2-s_1)+s_1) + (u \cdot  1_{u> s_2} - s_2)\big),
\end{align*}
where $(u \cdot 1_{u< s_1} -s_1) + (\hat v(s_2-s_1)+s_1) + (u \cdot  1_{u> s_2} - s_2) \in \mathcal{A}_{\nu}$, which is a contradiction with the minimality of $u$.

Letting $s_0$ as in the assumptions, we have just seen that for every $h >0$, $\frac{u_{[s_0 - h, s_0 + h]} - (s_0-h)}{2h}$ minimizes the total variation in $\mathcal{A}_{\nu_h}$  with
\begin{align*} 
\nu_h := \frac{  \int_{s_0-h  \ls u \ls s_0 + h} (u - (s_0-h))}{2h} + |E_{s_0 + h}| &= \frac{1}{2h} \int_{s_0 - h}^{s_0 + h} | \{ u > t\} \cap \{u \ls s_0 + h\} | \dd t +|E_{s_0 + h}| \\
&= \frac{1}{2h} \int_{s_0 - h}^{s_0 + h} | \{ u > t\} | \dd t = \frac{1}{2h} \int_{s_0 - h}^{s_0 + h} | E_s | \dd s.
\end{align*}
On the other hand, the total variation of $\frac{u_{[s_0 - h, s_0 + h]} - (s_0-h)}{2h}$ writes, using the coarea formula,
$$ \frac{1}{2h} \int_{s_0 - h}^{s_0 + h} \per(E_s).$$

Finally, let us assume that $1_{E_{s_0}}$ does not minimize total variation in $\mathcal{A}_{|E_{s_0}|}$. Then, there would exist $\eps > 0$ and $u_0 \in \mathcal{A}_{|E_{s_0}|}$ such that
$$ \per(E_{s_0}) \gs \TV(u_0) + \eps.$$
Since $s_0$ is a Lebesgue point, one can find $\delta >0$ such that for every $h \ls \delta$, 
$$ \left \vert \frac{1}{2h} \int_{s_0 - h}^{s_0 + h} \per(E_s) \dd s - \per(E_{s_0})  \right \vert \ls \frac{\eps}{10} \qquad \text{and} \qquad \left \vert \frac{1}{2h} \int_{s_0 - h}^{s_0 + h} | E_s | \dd s - |E_{s_0}| \right \vert \ls \frac{\eps}{10} .$$
Let $h\ls \delta$ and $B$ be a ball such that $\per(B) \ls \frac{\eps}{4 \Vert u_0 \Vert_\infty}$. There exists $\alpha$ such that the function $u_0 + \alpha 1_B$ satisfies
\begin{equation} \int u_0 + \alpha 1_B = \frac{1}{2h} \int_{s_0 - h}^{s_0 + h} | E_s | \dd s. \label{eq:LebMas} \end{equation}
Reducing $h$ if needed, one can enforce that $|\alpha| \ls 2 \Vert u_0 \Vert_\infty$.

Then,
$$ \TV(u_0 + \alpha 1_B) \ls \TV(u_0) + \alpha \per(B) \ls \per(E_{s_0}) - \eps + \alpha \per(B) \ls \frac{1}{2h} \int_{s_0 - h}^{s_0 + h} \per(E_s) \dd s - \frac{4 \eps}{10} ,$$
which contradicts the minimality of $u_{[s_0 - h, s_0+h]}$ and proves the claim.
\end{proof}

\subsection{Minimizers with connected level-sets}
In this subsection, we refine our analysis slightly, and show the existence of three-valued minimizers for Problem \ref{pr:1} with additional properties. We start with the following definition:
\begin{definition}\label{def:indecomp}A set of finite perimeter $A$ is called indecomposable, if there are no two disjoint finite perimeter sets $B, C$ such that $|B|>0$, $|C|>0$, $A = B \cup C$ and $\per(A)=\per(B)+\per(C)$.
\end{definition}
This notion is in fact a natural measure-theoretic sense of connectedness for sets for finite perimeter, for more information about it see \cite{AmbCasMasMor01}.
\begin{rem}By computing the Fenchel dual of Problem \ref{pr:1}, it can be seen that the non-zero level-sets of any solution are minimizers of the functional
$$E \mapsto \per(E) - \int_{\Omega \setminus \Omega_s} k, \text{ with } k \in L^2(\Omega \setminus \Omega_s).$$
This optimality property in turn implies lower bounds only depending on $k$ for the perimeter and mass of $E$, and in case it can be decomposed in the sense of Definition \ref{def:indecomp}, the same lower bounds also hold for each set in such a decomposition. In consequence, $E$ can only be decomposed in at most a finite number of sets. The proof of these statements relies heavily on the results of \cite{AmbCasMasMor01}, and is presented in \cite{ChaDuvPeyPoo17} for the unconstrained case,  and \cite{IglMerSch17} for the case with Dirichlet constraints, as used here.
\end{rem}
Assuming these results, one can simplify the level sets of solutions further:
\begin{thm}
There exists a minimizer for Problem \ref{pr:1} attaining exactly three values for which all non-zero level-sets are indecomposable.
\end{thm}
\begin{proof}
First, we consider the positive level-set and assume that it is decomposable in two sets $\Omega_1, \Omega_2$ as in Definition \ref{def:indecomp}. Then the corresponding minimizer $u$ can be written as 
 $$ u = \alpha (1_{\Omega_1} + 1_{\Omega_2}) -  \beta 1_{\Omega_-},$$
where $\alpha, \beta >0$. Consider a perturbation of $u$ of the form
$$u_h = (\alpha +h) 1_{\Omega_1} + (\alpha + k) 1_{\Omega_2} - (\beta + l) 1_{\Omega_-},$$
with $|h| \ls \alpha, |k| \ls \alpha$, and $|l| \ls \beta$. Then, since $\Omega_1 \cap \Omega_2 = \emptyset$, $u_h \in \BV_\diamond$ if and only if
$$ h|\Omega_1| + k |\Omega_2| - l |\Omega_-| = 0 \quad \text{and} \quad h|\Omega_1^s| + k |\Omega_2^s| - l |\Omega_-^s| = 0,$$
where $\Omega_i^s := \Omega_i \cap \Omega_s.$
These two equations lead to
$$ l = h \frac{|\Omega_1^s||\Omega_2| - |\Omega_1||\Omega_2^s|}{|\Omega_2||\Omega_-^s| - |\Omega_-||\Omega_2^s|} \quad \text{and} \quad k = h \frac{|\Omega_1^s||\Omega_-| - |\Omega_1||\Omega_-^s|}{|\Omega_2||\Omega_-^s| - |\Omega_-||\Omega_2^s|}.$$
Under our assumptions on $h, k, l, \Omega_1$ and $\Omega_2$, and since $1_{\Omega_1}+1_{\Omega_2}=1_{\Omega_1 \cup \Omega_2}$, the total variation of the perturbed function $u_h$ can be written as
\begin{align*}\TV(u_h) &= (\alpha+\min(h,k)) \per(\Omega_1 \cup \Omega_2) \\
&\quad +(h-k)^+ \per(\Omega_1) + (k-h)^+ \per(\Omega_2)+ (\beta + l) \per(\Omega_-)\\
&= (\alpha + h) \per(\Omega_1) + (\alpha + k) \per(\Omega_2) + (\beta + l) \per(\Omega_-).\end{align*}
Then, because $u$ is a minimizer of $\TV$, it follows that
$$ h \per(\Omega_1) + k \per(\Omega_2) + l \per(\Omega_-) \gs 0.$$ 
Since the left hand side and $k,l$ are linear in $h$, one can replace $h$ by $-h$ and obtain 
$$ h \per(\Omega_1) + k \per(\Omega_2) + l \per(\Omega_-) = 0$$ 
which shows that $u_h$ is also a minimizer. Now since we have 
$$\beta = \alpha \,\frac{|\Omega_1|+|\Omega_2|}{|\Omega_-|} \text{ and } l = \frac{h|\Omega_1|+k|\Omega_2|}{|\Omega_-|},$$
one can choose $h$ such that $h=-\alpha$ or $k = -\alpha$ without violating $|l|\leq \beta$, and therefore produce a minimizer whose positive part is either 
$\Omega_2$ or $\Omega_1$, respectively. We proceed similarly for the negative part and therefore obtain an indecomposable negative level-set.
\end{proof}
\begin{rem}
In the above proof, through an adequate choice of components for deletion, one can even obtain simply connected level sets. The measure-theoretic notion corresponding to simple connectedness is defined in \cite{AmbCasMasMor01} to be boundedness of the connected components of the complement of the set, these connected components having been defined through indecomposability. For example, assuming that $\Omega_2$ is fully enclosed in $\Omega_-$ (that is if $\partial \Omega_2 \cap \partial \Omega_- = \partial \Omega_2$), 
 then the variation of $u_h$ can also be written
 $$ \TV(u_h) = (\alpha + h) \per(\Omega_1) + (\alpha + k + \beta + l) \per(\Omega_2) + (\beta + l) (\per(\Omega_-) - \per(\Omega_2),$$
 which is linear in $h$ as long as $k \gs -\alpha - \beta - l$. The equality case in this last constraint corresponds to joining $\Omega_2$ to $\Omega_-$, and avoiding creating a ``hole'' in $\Omega_-$ by the procedure mentioned above (which replaces $ \alpha 1_{\Omega_2}$ by zero). Clearly, this procedure can also be performed for the positive level set, and in fact the ``holes'' to be deleted could also be connected components of the zero level set. Therefore, a solution in which both the positive and negative level set are simply connected can be obtained.
\end{rem}
\begin{rem}
The intuition behind these last results is that, like in the proof of Theorem \ref{thm:3val}, 
the constraints of the problem are linear with respect to the values, and the total variation is also linear 
as long as the signs of the differences of values at the interfaces do not change. In particular, the points at 
which the topology of the level sets changes are situations in which these signs change (that is, the values of 
two adjacent level sets are equal).
\end{rem}

\section{Numerical scheme and results}\label{sec:numerics}
We now turn our attention to the numerical computation of solutions to the eigenvalue for Problem \ref{pr:1}. At first, for simplicity, we limit ourselves to the case (considered in \cite{FriIglMerPoeSch17}) in which the velocities are assumed constant on the whole $\Omega_s$. That is, the problem considered is minimization of the total variation in the space 
\begin{equation}\label{eq:bvdo}\BVdo:=\left\{u \in \BV(\R^2) \ \middle \vert \ \int_\Omega u = 0\,, \; u \equiv 1 \text{ in } \Omega_s, u \equiv 0 \text{ in } \R^2 \setminus \Omega\right\}\,,\end{equation}
where the constraint $u \equiv 1 \text{ in } \Omega_s$ corresponds to \eqref{eq:normalize} under this simplification.

This restriction corresponds to the case in which either $\Omega_s$ is connected, so that there is only one solid particle, or all the particles are constrained to move with the same velocity. In Section \ref{sec:multiNum} we point out the required modifications for the multi-particle case and present a variety of computed examples.

To compute a minimizer of $\TV$ in $\BV_{\diamond,1}$, we use a standard primal dual algorithm \cite{ChaPoc11}. The constraint $\int_{\Omega} v = 0$ is enforced through a scalar Lagrange multiplier $q$, whereas the conditions $v = 0$ on $\partial \Omega$ and $v = 1$ on $\Omega_s$ are encoded as indicator functions. Our discretization of choice is finite differences on a rectangular grid $\{1,\ldots, m\}\times\{1,\ldots n\}$, where in this whole section, for simplicity, we assume that $n=m$ and $\Omega \Subset (0,1)^2$. This leads to a saddle point problem of the form
\begin{equation}\label{eq:discmin}
\min_{v \in X} \max_{\substack{p \in X^4 \\ \mkern-10mu q \in \mathbb R}} \rchi_{C^n}(v) + \sum_{i,j} \left[\, (\nabla v)^{ij} \cdot p^{ij} - \rchi_{\{|\cdot|_{\infty} \leq 1\}}(p^{ij})-q v^{ij}\,\right].
\end{equation}
Here, $X=\R^{n^2}$ denotes the space of real-valued discrete functions on the square grid $G^n=\{1,\ldots, n\}\times\{1,\ldots n\}$. Since we use Dirichlet boundary conditions, the grid encloses the physical domain. The corresponding constraint set is then
\begin{equation}\label{eq:cnsingle}C^n := \left\{v \in X \mid v=0 \text{ on } G^n \setminus \Omega^n, v=1 \text{ on }\Omega^n_s\right\},\end{equation}
where $\Omega^n$ and $\Omega^n_s$ denote the parts of the grid corresponding to $\Omega$ and $\Omega_s$ respectively (note that to correctly account for perimeter at the boundary we must have $\Omega^n \subset \{2,\ldots, n-1\}\times\{2,\ldots n-1\}$). The indicator function (in the convex analysis sense) of a set $A$ is denoted by $\rchi_A$, so that $\rchi_A(x)=0$ if $x \in A$, and $+\infty$ otherwise. $\nabla$ stands for a suitable discrete gradient, whose choice we now discuss.

\subsection{Discretization}
We discretize the problem using the ``upwind'' scheme of \cite{ChaLevLuc11} which has the advantage of carrying a high 
degree of isotropy. The discrete velocity is denoted by $v^{ij}$, and we use the signed gradient $(\nabla v)^{ij}$ introduced in \cite{ChaLevLuc11}, containing separate components for forward and backward differences with opposite signs:
\begin{equation}\label{eq:grad}
\begin{aligned}(\nabla v)^{ij} & := \left(v^{i+1,j}-v^{i,j}, \  v^{i-1,j} - v^{i,j},\ v^{i,j+1} - v^{i,j},\ v^{i,j-1} - v^{i,j}\right) \\
&=: \left( (\nabla v)^{ij}_{1,+}, \ (\nabla v)^{ij}_{1,-}, \ (\nabla v)^{ij}_{2,+}, \ (\nabla v)^{ij}_{2,-} \right)
\end{aligned}
\end{equation}
therefore, at each grid point $(i,j) \in G^n$ the signed gradient and its corresponding multiplier variables $\nabla v^{ij}, p^{ij} \in (\R^2)^2$. We note that to compute the gradient when any of the indices is $1$ or $n$ one needs to extends the functions outside the grid, but for the problem at hand any choice will do, since $\Omega^n$ never touches the boundary of the grid.

For us it is important to use a discretization that takes into account derivatives in all coordinate directions equally, since we aim to resolve sharp geometric interfaces that are not induced by a regularization data term. Figure \ref{fig:badcross} contains a comparison with the results obtained when using forward differences. In that case, the geometry of the interfaces is distorted according to their orientations, a phenomenon which is minimized in the upwind scheme. Using centered differences is also not adequate, since the centered difference operator has a nontrivial kernel and our solutions are constant in large parts of the domain.

\begin{figure}[htbp]
    \begin{center}
       \subfigure{
           \label{fig:cross-bc}
           \includegraphics[width=0.29\textwidth]{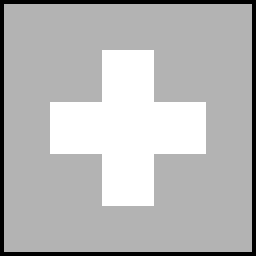}
       }      
       \subfigure{
           \label{fig:cross-fd}
           \includegraphics[width=0.29\textwidth]{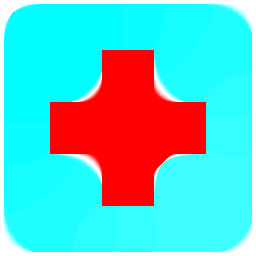}
       }
       \smallskip     
       \subfigure{
           \label{fig:cross-good}
           \includegraphics[width=0.29\textwidth]{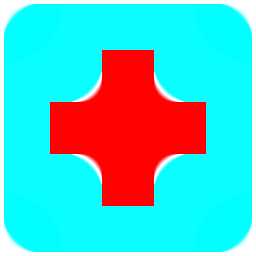}
       }
       \subfigure{
           \label{fig:cross-good-int}
           \includegraphics[width=0.45\textwidth]{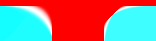}
       }
       \subfigure{
           \label{fig:cross-fd-int}
           \includegraphics[width=0.45\textwidth]{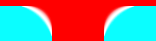}
       }
   \end{center}
   \caption{Results (contrast-enhanced images) with different discretizations. Top row: Boundary conditions, result with only forward differences and with the chosen upwind scheme. Bottom row: Detail of the interfaces in both cases. For the upwind scheme, the resulting interfaces depend less strongly on their orientation, and there are two flip symmetries.}
  \label{fig:badcross}
\end{figure}

\subsection{Convergence of the discretization}
It is well-known that the standard finite difference discretizations of the total variation converge, in the sense of 
$\Gamma$-convergence with respect to the $L^1$ topology \cite{ChaLevLuc11},
where the discrete functionals are appropriately defined for piecewise constant functions. We now aim to demonstrate that the chosen discretization and penalization scheme still converges and correctly accounts for the boundary conditions in the limit. We introduce, for each $(i,j) \in \set{1,\ldots,n-1}^2$,
$$R_{ij}^n := \frac 1n \left( i-\frac 12, i+ \frac 12 \right) \times \left( j-\frac 12, j+ \frac 12 \right).$$
First, we need to decide which constraint to use in the discrete setting. We denote by 
$$E - B\left(\frac 1n \right) := \left\{x \in E \; \middle 
\vert\; d(x,\partial E) > \frac 1n \right\},$$ 
Our choice is to take
$$\Omega_s^n := \bigcup_{R_{ij}^n \subset \Omega_s - B(\frac 1n)} R_{ij}^n$$
whereas
$$\Omega^n := [0,1]^2 \setminus \left( \bigcup_{R_{ij}^n \subset ([0,1]^2 \setminus \Omega) - B(\frac 1n)} R_{ij}^n \right),$$
such that the discrete constraints are less restrictive than the continuous ones (see Figure \ref{fig:discretedomain}) and
\begin{equation}\label{eq:discompact}\overline{\Omega_s^n} \Subset \Omega_s, \quad  \overline{[0,1]^2 \setminus \Omega^n} \Subset [0,1]^2 \setminus \Omega.\end{equation}
We define $\TV^n$ as in \cite{ChaLevLuc11}, when the function is piecewise constant on the $R_{ij}^n$ and $+\infty$ otherwise.
$$\TV^n := \frac{1}{n^2} \sum_{i,j} |\nabla v^{ij} \vee 0|$$
with $\nabla v^{ij} \vee 0$ denotes the positive components of $\nabla v^{ij}$, which was defined in \eqref{eq:grad}, therefore picking only the `upwind' variations. The norm is computed using the inner product in $\mathbb R^{2 \times 2}$.

\begin{figure}[htbp]
    \begin{center}
		\includegraphics[width=.6\textwidth]{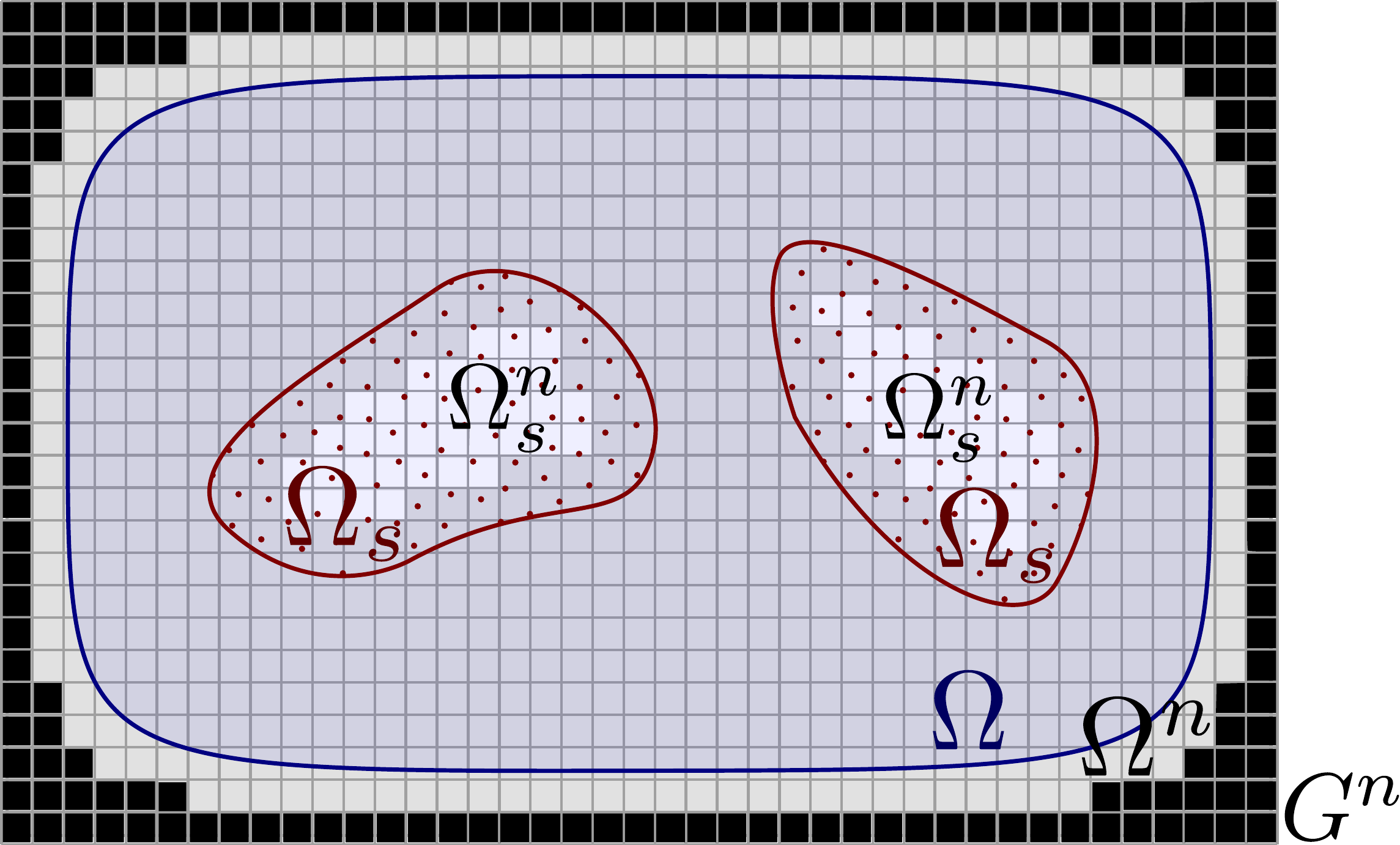}
    \end{center}
   \caption{Discretization of the domain and constraints: The discrete grid encloses $\Omega$, and discrete regions are only constrained if they are compactly contained in the corresponding continuous ones. Here, grey squares have free values while the black and white ones are fixed.}
  \label{fig:discretedomain}
\end{figure}

We first prove the following lemma, which states that the continuous total variation may be computed with multipliers with positive components, mimicking the discrete definition.
\begin{lem}\label{lem:TVpos}
Let $v \in \BV(\R^d)$ and $\Omega \subset \R^d$ open. Then, $\TV(v,\Omega) = \TV^+(v,\Omega)$, where
\begin{equation}\TV^+(v,\Omega):=\sup \left\{\int_\Omega v \cdot \div p - v \cdot \div q \ \middle \vert \  p,q \in \mathcal C_0^1(\Omega, \R^d),\, |p|^ 2+|q|^2\ls 1,\, p,q \gs 0 \right \}. \label{eq:TVpos} \end{equation}
\end{lem}
\begin{proof}
We recall that
\begin{equation}\TV(v,\Omega) = \sup\left \{ \int_\Omega v \cdot \div p \ \middle \vert \ p \in \mathcal C_0^1(\Omega,\R^d),\, |p| \ls 1 \right \}.\label{eq:TVnormal} \end{equation}
Let $p,q$ be admissible in the right hand side of \eqref{eq:TVpos}. Then we notice that $p-q$ is also admissible in \eqref{eq:TVnormal}, because $p,q$ being componentwise positive implies
$$|p-q|^2=|p|^2+|q|^2-2\, p \cdot q\ls 1-2\, p \cdot q \ls 1,$$
and since $\div (p-q)=\div p - \div q$ we have
$$\TV^+(v,\Omega) \ls \TV(v,\Omega).$$

To prove the reverse inequality, let $\eps > 0$ be arbitrary and $p_\eps \in \mathcal C_0^1(\Omega,\R^d)$ with $|p_\eps| \ls 1$ such that 
$$\TV(v,\Omega) - \int_{\Omega} \sum_{j=1}^d v_j \div (p_\eps)_j < \eps,$$
which we can write (renaming $p_\eps$ to its additive inverse, for convenience) as
\begin{equation}\label{eq:prodineq}
\int_{\Omega} \left(1 - p_\eps \cdot \frac{\d (\nabla v)}{\d|\nabla v|} \right) \d|\nabla v| < \eps.
\end{equation}
Noting that $|p_\eps| \ls 1$ and $\frac{\d \nabla v}{\d|\nabla v|} \ls 1$, the last inequality implies 
(since for $|\mu|,|\nu| \ls 1$, $|\mu - \nu|^2 \ls 2 - 2 \mu \cdot \nu$) as
\begin{equation}\label{eq:distineq}\int_{\Omega} \frac 12 \left \vert  p_\eps- \frac{\d \nabla v}{\d|\nabla v|}  \right \vert^2 \dd|\nabla v| < \eps.\end{equation}
Notice that we may write this integral, since the function $\frac{\d \nabla v}{\d|\nabla v|}$ is a Radon-Nikodym derivative, in principle only in $L^1(\Omega, |\nabla v|)$, but its modulus is $1$ for $|\nabla v|$-almost every point \cite[Corollary 1.29]{AmbFusPal00}, so it is also in $L^2(\Omega, |\nabla v|)$. Now, by \eqref{eq:prodineq} and the Cauchy-Schwarz inequality we have
\begin{equation}\label{eq:closetoone}
\begin{aligned}\int_{\Omega}1-|p_\eps|^2 \dd|\nabla v|&=\int_{\Omega}\big(1-|p_\eps|\big)\big(1+|p_\eps|\big)\dd|\nabla v| \leq 2\int_{\Omega}1-|p_\eps| \dd|\nabla v|\\
&\leq 2 \int_{\Omega} \left( 1 - p_\eps \cdot \frac{\d \nabla v}{\d|\nabla v|} \right) \d|\nabla v| <2\eps,
\end{aligned}
\end{equation}

Now we replace the components $(p_\eps)_{ij}$ by $(\tilde p_\eps)_{ij}$ which are smooth, coincide with 
$(p_\eps)_{ij}$ out of $\{|(p_\eps)_{ij}| < \sqrt{\eps} \}$, that satisfy
$$|(\tilde p_\eps)_{ij}| \ls |(p_\eps)_{ij}|$$
and such that $\{(\tilde p_\eps)_{ij} = 0\}$ is the closure of an open set:
One can for example choose 
$$0 < \alpha < \sqrt{\eps}$$
and define a smooth nondecreasing function $\psi_\alpha : \R \to \R$ such that $\psi_\alpha(t) = t$ for $|t| \gs \alpha$, $|\psi_\alpha(t)| \leq |t|$ and $\psi_\alpha (-\alpha/2, \alpha/2)=\{0\}$
to define
$$(\tilde p_\eps)_{ij} := \psi_\alpha \circ (p_\eps)_{ij}.$$
Thus we have $|\tilde p_\eps| \ls 1$ and $|(\tilde p_\eps)_{ij} - (p_\eps)_{ij} | \ls \sqrt{\eps}$, and taking into account \eqref{eq:prodineq} we obtain
\begin{equation}\label{eq:tildedistineq}\begin{aligned}\Biggl( \int_{\Omega} \left \vert  \tilde p_\eps- \frac{\d \nabla v}{\d|\nabla v|}  \right \vert^2 \dd|\nabla v| \Biggr)^{\frac 12} &\ls \left( \int_{\Omega} \left \vert p_\eps - \frac{\d \nabla v}{\d|\nabla v|}  \right \vert^2 \dd|\nabla v| \right)^{\frac 12} + \left( \int_{\Omega} |\tilde p_\eps - p_\eps |^2 \dd|\nabla v| \right)^{\frac 12} \\
&\ls C \sqrt{\eps} \left( 1 + |\nabla v|(\Omega)\,\right).
\end{aligned}\end{equation}

Furthermore, using \eqref{eq:closetoone} and the definition of $\tilde p_\eps$ we obtain the estimate
\begin{equation*}\label{eq:tildeclosetoone}
\int_{\Omega}1-|\tilde p_\eps|^2 \dd|\nabla v| = \int_{\Omega}1-|p_\eps|^2 \dd|\nabla v| + \int_{\Omega} | p_\eps |^2 - |\tilde p_\eps |^2 \dd|\nabla v| \leq 2\eps + 4\eps|\nabla v| < C\eps(1+|\nabla v|),
\end{equation*}
which ensures, writing $1-\mu:\nu = \frac 12 (1-|\mu|^2 + 1- |\nu|^2 + |\mu - \nu|^2)$ and by \eqref{eq:tildedistineq} that
\begin{equation}\label{eq:tildealmostaligned}\left \vert \int_{\Omega} \left(\tilde p_\eps : \frac{\d \nabla v}{\d|\nabla v|} -1 \right) \dd|\nabla v| \right \vert \ls C \eps\left( 1 + \big(\ 1 + |\nabla v|(\Omega)\, \big)^2\right).\end{equation}
Now, we notice that having fattened the level-set $\{(\tilde p_\eps)_{ij}=0\}$, we can write 
$$ (\tilde p_\eps)_{ij} = \left[(\tilde p_\eps)_{ij}\right]^+ - \left[(\tilde p_\eps)_{ij}\right]^-$$ where both quantities are smooth.
Writing similarly
$$\tilde p_\eps = \tilde p_\eps^+ - \tilde p_\eps^-$$
with $\tilde p_\eps^\pm$ are smooth and have only positive components, we note that $(\tilde p_\eps^+,\tilde p_\eps^-)$ are admissible in the right hand side of \eqref{eq:TVpos}, 
so that \eqref{eq:tildealmostaligned} implies
$$\TV^+(v,\Omega) \gs \TV(v,\Omega) -C(v)\eps.$$
Letting $\eps \to 0$, we conclude.
\end{proof}
We can now prove Gamma-convergence of the discrete problems, implying convergence of the corresponding minimizers.
\begin{thm}\label{thm:conv}
$$ \TV^n + \rchi_{C^n} \xrightarrow{\Gamma-L^1} \TV + \rchi_C$$
where $$C^n := \{v=1 \text{ on } \Omega_s^n, 0 \text{ on } [0,1]^2 \setminus \Omega^n)$$ and $$C := \{v=1 \text{ on } \Omega_s, 0 \text{ on } [0,1]^2 \setminus \Omega)\}.$$
\end{thm}
\begin{proof}
First, we study the $\Gamma$-liminf and assume that $v_n \to v$ in $L^1$. Notice that we can write $\TV^n(v_n)$ as a dual formulation
$$\TV^n(v_n) = \sup\left\{ v_n \cdot \div^n(p) \, \mid \, p:G \to (\R^2)^2 \right\}$$
where $\div^n p \in \R^2$ is the signed divergence corresponding to \eqref{eq:grad}, and defined by 
\begin{align*}(\div^n(p))^{ij} := &(p_{1,+})^{i,j} - (p_{1,+})^{i-1,j} + (p_{1,-})^{i,j} - (p_{1,-})^{i+1,j} \\&(p_{2,+})^{i,j} - (p_{2,+})^{i,j-1} + (p_{2,-})^{i,j} - (p_{2,-})^{i,j+1}.\end{align*}
This is obtained easily by a (discrete) integration by parts in the expression
$$|\nabla v \vee 0| = \sup_{\substack{|p| \ls 1 \\ p_i \gs 0}} \nabla v \cdot p.$$

Now, we note that every $p : G \to (\R^2)^2$ can be viewed as the discretization of some smooth function $\overline p:[0,1]^2 \to (\R^2)^2$, for example stating 
$$ p^{ij} = \fint_{R_{ij}} \overline p.$$
As a result, one can write
$$\TV^n(v_n) = \sup\left\{ v_n \cdot \div^n(\overline p) \, \middle\vert \, \overline p\in \mathcal C_0^1\left([0,1]^2, (\R^2)^2\right), \ | \overline p| \ls 1, \overline p \gs 0 \ \right\}.$$
It is well known that for a smooth function $\overline p$, the quantity $\div^n \overline p$ converges to
\begin{align*}\div \overline p & = \div(\overline p_{1,1},\overline p_{2,1})+ \div(-\overline p_{1,2},-\overline p_{2,2}).
\end{align*}
Therefore, using Lemma \ref{lem:TVpos} we get
$$\TV^+(v) = \TV(v) \ls \liminf \TV^n(v_n).$$
For $\rchi_{C^n}$, let us first assume $\rchi_{C}(v) = + \infty$, that is either $v \not \equiv 0$ on $[0,1]^2 \setminus \Omega$ or $v \not \equiv 1$ 
on $\Omega_s$. If the latter holds, then for $\eps$ small enough, 
$\Omega_s \cap \left(\{v > 1+2\eps\} \cup \{v < 1- 2 \eps \} \right)$ has positive measure and thanks to the 
$L^1$ convergence of $v_n$,
$$\Omega^n_s \cap \left( \{v_n >1+ \eps \} \cup \{v_n < 1- \eps \} \right)$$
must have a positive measure for $n$ big enough. That implies $\rchi_{C^n}(v_n) = + \infty$ and the $\Gamma$-liminf inequality is trivially true. If $\rchi_{C}(v) < \infty$, then $\rchi_{C}(v) =0$ and the inequality is also true since $C^n \subset C$.

Let now $v \in \BV ((0,1)^2)$. For the $\Gamma$-limsup inequality we want to construct a sequence 
$v_n \to v$ such that $$\TV(v)  + \rchi_C(v) \gs \limsup \TV^n(v_n) + \rchi_{C^n}(v_n).$$ If $v \notin C$, any $v_n \to v$ gives the inequality. If $v \in C$, then we first introduce
$$v_\delta = \psi_\delta  \ast v$$
where $\psi_\delta$ is a convolution kernel with width $\delta$.

Then, $\TV(v_\delta) \to \TV(v)$ (\cite[Theorem 1.3]{AnzGia80}, noticing that $v$ is constant around $\partial [0,1]^2$) and, thanks to \eqref{eq:discompact}, if $\delta \ls \frac 1n,$ we have $\rchi_{C^n}(v_\delta) = 0.$

We define $v_{\delta,n}$ by
$$(v_{\delta,n})^{ij} = \fint_{R_{ij}^n} v_\delta, $$
that satisfies $\rchi_{C^n}(v_{\delta,n}) = 0$, and compute
$$ \frac{ \left | (v_{\delta,n})^{i+1,j}-(v_{\delta,n})^{i,j} \right |}{n} = \frac 1n \left |\fint_{R_{ij}^n} v_\delta(x+\frac 1n,y,z) - v_\delta(x,y) \right | \gs \inf_{R_{ij}^n \cup \big( R_{ij}^n+(\frac 1n,0) \big) } \vert \partial_x v_\delta \vert . $$
Then since $v_\delta \in C^1$, it is clear that the right hand side converges to $|\partial_x v_\delta|$.
Note that in the 'upwind' gradient of a smooth function, only one term by direction can be active, then it is also true for 
$v_{\delta,n}$ if $n$ is large enough and therefore $\TV^n(v_{\delta,n}) \to \TV(v_\delta).$
By a diagonal argument on $\delta$ and $n$, we conclude.
\end{proof}

\subsection{Single particle results}
\label{sec:BV}
In this section, we again restrict ourselves to the case in which there is either only one particle, or the particles are constrained to move with the same velocity. 
 
In \cite{FriIglMerPoeSch17}, it is shown analytically that the minimizers of $\TV$ over the set $\BV_{\diamond, 1}$ defined in \eqref{eq:bvdo} have level-sets that minimize some geometrical quantities. In particular, Theorem 4.10 shows that there exists a minimizer of the form
 $$ u_0 := 1_{\Omega_1} - \lambda 1_{\Omega_-}$$
 where $\Omega_-$ is the maximal Cheeger set of $\Omega \setminus \Omega_s$, and $\Omega_1$ is a minimizer of $$E \mapsto P(E)+\frac{P(\Omega_-)}{|\Omega_-|} |E|$$ over $E \supset \Omega_s$.
 
Unfortunately, determining Cheeger sets analytically is only possible in a very narrow range of sets, which makes useful the numerical computation of minimizers. We present two examples of the output of the numerical method for \eqref{eq:discmin} with the constraint \eqref{eq:cnsingle}. First, we consider the ``Pacman'' shaped $\Omega_s$ within again a square $\Omega$; see Figure \ref{fig:pacman} (left). This example induces both asymmetry (left-right) and non-convexity of $\Omega_s$ which is showed in \cite{FriIglMerPoeSch17} to influence the geometry of the minimizer. The solution is shown in the central panel of Figure \ref{fig:pacman} and the right-hand panel shows a histogram of the solution. 
\begin{figure}[htb]
    \begin{center}
       \subfigure{
           \label{fig:pacman-bc1}
           \includegraphics[width=0.225\textwidth]{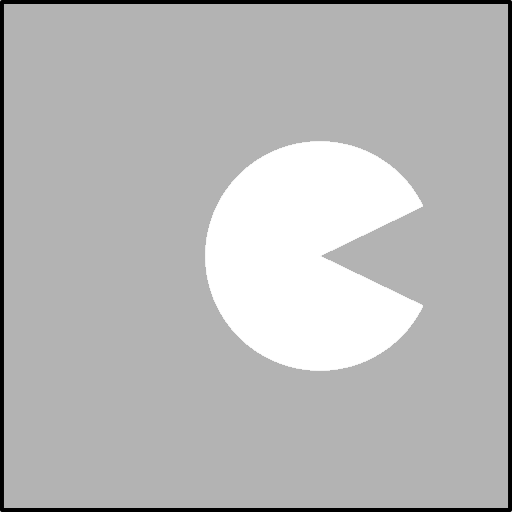}
       }
       \subfigure{
          \label{fig:pacman-result1}
          \includegraphics[width=0.225\textwidth]{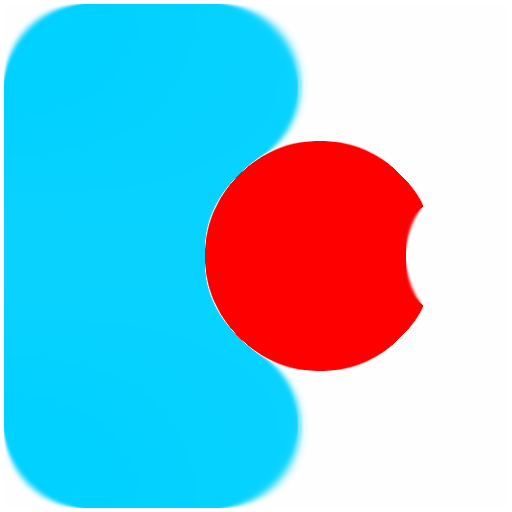}
       }
       \subfigure{
           \label{fig:pacman-bc2}
           \includegraphics[width=0.225\textwidth]{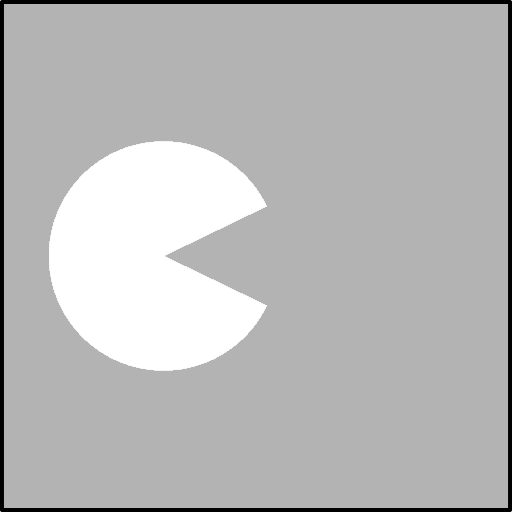}
       }
       \subfigure{
          \label{fig:pacman-result2}
          \includegraphics[width=0.225\textwidth]{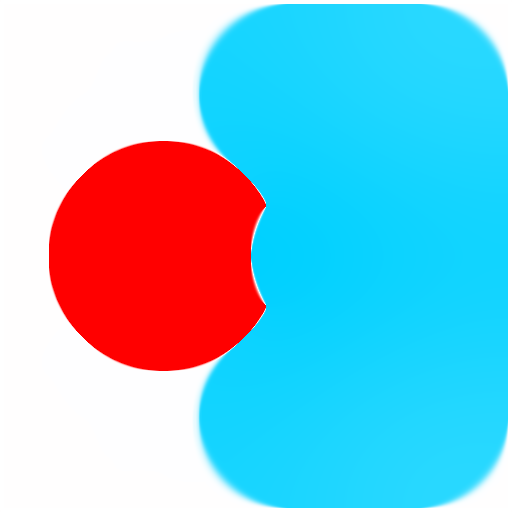}
       }       
   \end{center}
   \caption{Numerical results with a single particle, illustrating the results of \cite{FriIglMerPoeSch17}. On the left of the two subfigures, the free part $\Omega^n \setminus \Omega^n_s$ of the computational domain $G^n$ is in gray, and the particles $\Omega^n_s$ are white. The minimizers are on the right, where the blue colour represents the negative values and the red colour, the positive ones. More precisely, the left result has nonzero values $\{-2.38, 7.41\}$ whereas the right result has $\{-2.35, 7.41\}.$ The corresponding computed critical yield numbers are is $Y_c=0.0576$ and $Y_c=0.0596$, with $\Omega$ having side length $1$.}
  \label{fig:pacman}
\end{figure}

The second example concerns the geometry depicted in Figure \ref{fig:dumbbell} (top panel), in which $\Omega_s$ denotes the two L-shaped regions in the white dumbbell-shaped domain $\Omega$. By giving a close look, it is clear that there is a Cheeger set of $\Omega \setminus \Omega_s$ in each half of the domain, which implies the non uniqueness of the minimizer. The question is which solution the computations will converge to. Figure \ref{fig:dumbbell} (lower, left and right) show that different minimizers are selected numerically, in this case by using different numerical resolution.  

\begin{figure}[htb]
    \begin{center}
       \subfigure{
           \label{fig:dumbbell-bc}
           \includegraphics[width=0.4\textwidth]{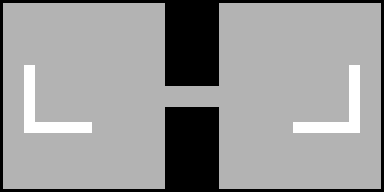}
       }\\
       \subfigure{
          \label{fig:dumbbell-result}
          \includegraphics[width=0.4\textwidth]{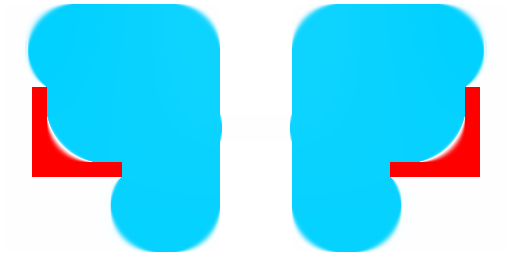}
       }
       \subfigure{
           \label{fig:dumbbell-histogram}
           \includegraphics[width=0.4\textwidth]{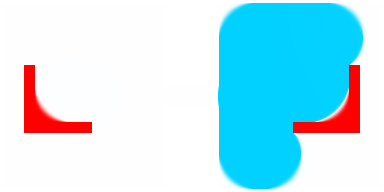}
       }
   \end{center}
   \caption{Boundary conditions and results computed at two different resolutions, in a situation when uniqueness of minimizers of $\TV$ in $\BV_{\diamond, 1}$ is not expected \cite{KawLac06, FriIglMerPoeSch17}. The nonzero values are $\{-4.45,53.0\}$ for the left result and $\{-8.67, 53.8\}$ for the right one. In both cases, $Y_c = 0.087$, where the longest side of $\Omega$ is $1$.  }
  \label{fig:dumbbell}
\end{figure} 

\subsection{Several particles}\label{sec:multiNum}
We now extend the numerical scheme of to optimize also over the velocities $\gamma_i$ on each component $\Omega_s^i$. The corresponding problem is again the minimization \eqref{eq:discmin}, but with the new constraint set
\[C^n := \left\{v \in X \ \middle \vert \ v=0 \text{ on }G^n \setminus \Omega^n,\; v \text{ constant on }(\Omega^n_s)^i, \;\frac{1}{|\Omega^n_s|} \sum_{\Omega^n_s} v = 1\right\}.\]
Here, $(\Omega^n_s)^i$ denotes the $i$-th component of the discrete domain, corresponding to $\Omega_s^i$. The set $C^n$ is the discrete counterpart to the set $\BV_\diamond$ used in sections \ref{sec:relax} and \ref{sec:geomsols}.

We give several examples that illustrate the behavior of $\TV$-minimizers in $\BV_\diamond$ with a disconnected $\Omega_s$. Figure \ref{fig:bridges} shows the influence of the positions of particles with respect to each other and to the boundary, which might lump up in different configurations. Figure \ref{fig:MultpartGeneric} shows two generic situations: \ref{fig:MultpartGeneric1}, the flowing part is concentrated around one connected component of $\Omega_s$ whereas on \ref{fig:MultpartGeneric2}, it is concentrated around the whole $\Omega_s.$

\begin{figure}[htb]
     \begin{center}
     \mbox{} \hfill
     \raisebox{-0.5\height}{\includegraphics[width=0.32\textwidth]{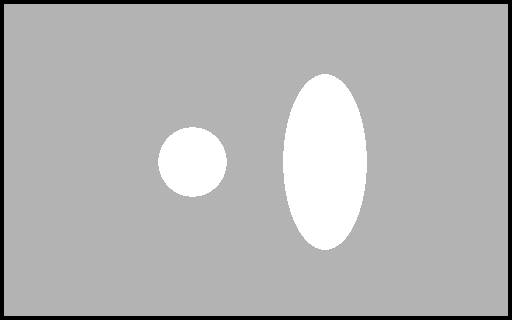}}
     \hfill
     \raisebox{-0.5\height}{\includegraphics[width=0.32\textwidth]{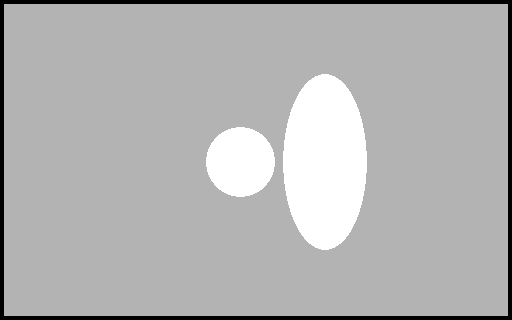}}
     \hfill
     \raisebox{-0.5\height}{\includegraphics[width=0.32\textwidth]{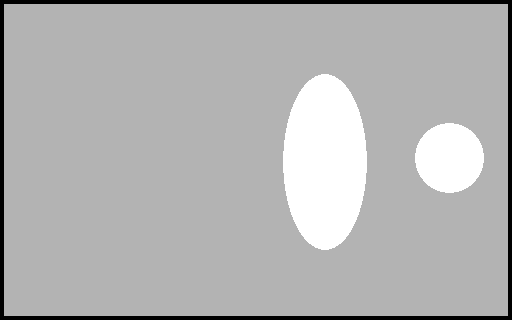}}
     \hfill \\\vspace{0.1cm}
     \hfill
     \raisebox{-0.5\height}{\includegraphics[width=0.32\textwidth]{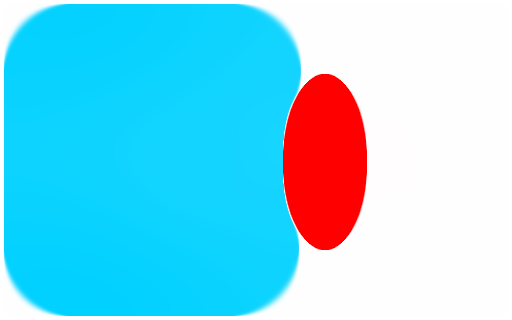}}
     \hfill
     \raisebox{-0.5\height}{\includegraphics[width=0.32\textwidth]{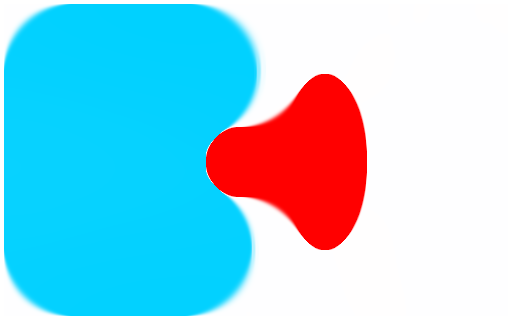}}
     \hfill
     \raisebox{-0.5\height}{\includegraphics[width=0.32\textwidth]{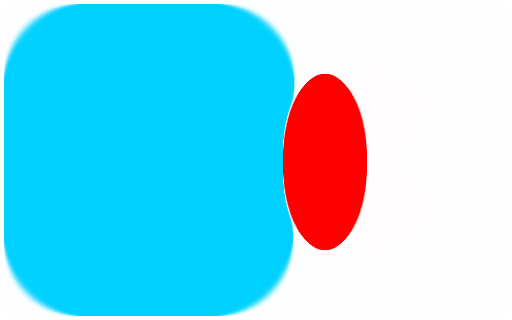}}
     \hfill
     \mbox{}
     \end{center}
\caption{Examples of minimizers of $\TV$ in $\BV_\diamond$ for several particles. The top row represents the boundary conditions. The computed minimizers are depicted below, where the blue colour represents the negative values and the red colour, the positive ones. The nonzero values are $\{-3.40, 23.7\}$, $\{-4.27, 17.1\}$ and $\{-3.21, 22.9\}$ respectively, whereas the corresponding critical yield numbers are $Y_c=0.0378$, $Y_c=0.0383$ and $Y_c=0.0396$, again when the longest side of $\Omega$ is $1$.}
\label{fig:bridges}
\end{figure}

\begin{figure}[htb]
 \begin{center}
  \subfigure{ 
  \label{fig:MultpartGeneric1}
  \includegraphics[width = 0.23 \textwidth]{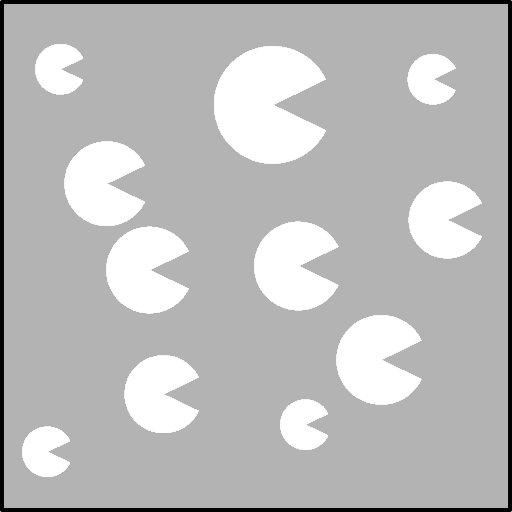}
  \includegraphics[width = 0.23 \textwidth]{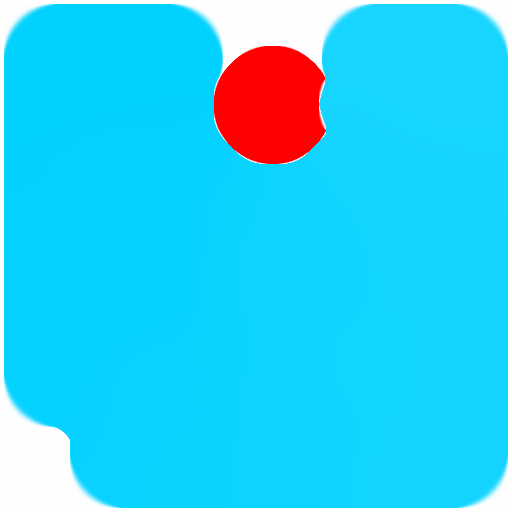}
  }
  \subfigure{ 
  \label{fig:MultpartGeneric2}
    \includegraphics[width = 0.23 \textwidth]{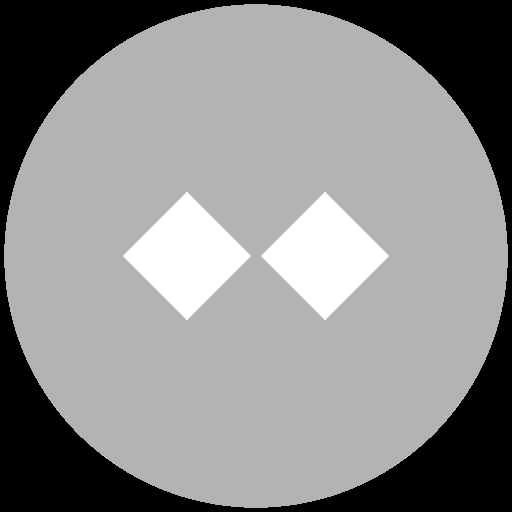}
    \includegraphics[width = 0.23 \textwidth]{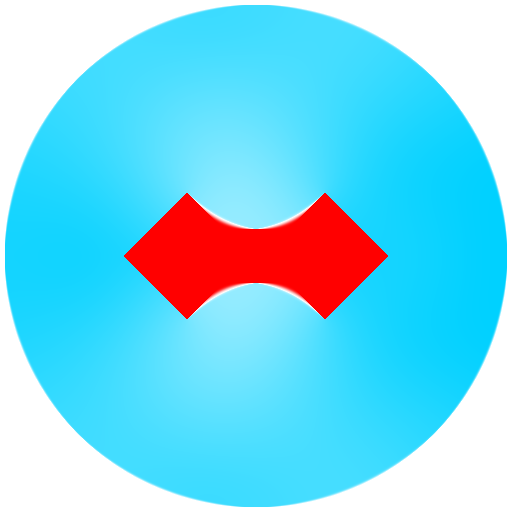}
  }
 \end{center}
\caption{Two examples of minimizing $\TV$ in $\BV_\diamond$ for several particles. Here, the nonzero values are $\{-1.78, 33.8\}$ and $\{-2.21, 16.3\}$ and $Y_c=0.0324$ (the length of a side of $\Omega$ being $1$) and $Y_c=0.0344$ (the diameter of $\Omega$ being $1$) respectively. For the left result, since the magnitude of the negative values is much smaller than that of the positive ones, their color has been rescaled.}
\label{fig:MultpartGeneric}
\end{figure}

We also give an example where uniqueness of the minimizer is not expected. In Figure \ref{fig:MultNonUniq}, we consider a grid of circular particles in a square. It is easy to see analytically that any subset of the particles can be chosen as positive part of the minimizer. We present two computations at different numerical resolutions that pick two different subsets.

Since the solutions we compute correspond to limit profiles of the original flows (Theorem \ref{thm:conv}), the results presented both here and in Section \ref{sec:3level} mean that near the stopping regime $Y \to Y_c$ the transition between yielded and unyielded regions of the fluid typically happens closer and closer to the particle boundaries and the domain boundaries. This is consistent with the Cheeger set interpretation of the buoyancy case (which was already present in \cite{FriIglMerPoeSch17}) and the many previous works on non-buoyancy cases (\cite{MosMia66, Hild2002}, for example).

\begin{figure}[htb]
 \begin{center}
 \includegraphics[width = 0.32 \textwidth]{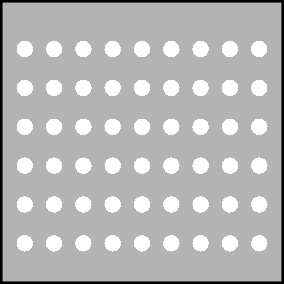}
  \includegraphics[width = 0.32 \textwidth]{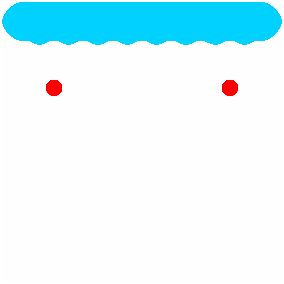}
    \includegraphics[width = 0.32 \textwidth]{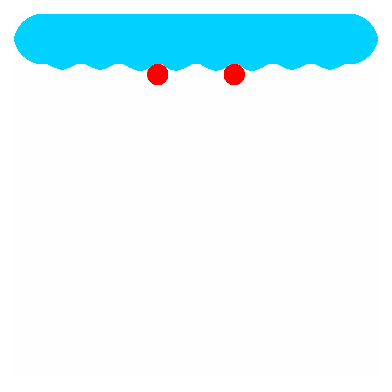}
 \end{center}
\caption{Numerical computation of a minimizer at two different resolutions when uniqueness is not expected. Here, since the magnitude of the negative values is much smaller than that of the positive ones, their color has been rescaled.}
\label{fig:MultNonUniq}
\end{figure}

\subsection{A random distribution of small particles}
We also present two examples of random distribution of square particles in a bigger square. Figure \ref{fig:randpart} shows the same number of particles distributed in two different ways and the corresponding minimizers. This example shows that the yield number depends strongly on the geometry of the problem, not only on the ratio solid/fluid. An interesting problem would be to investigate the optimal distribution to maximize/minimize this yield number.

\begin{figure}
 \begin{center}
 \includegraphics[width = 0.4 \textwidth]{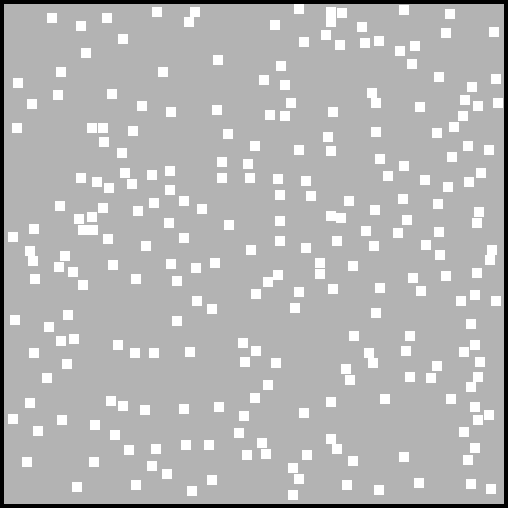}
 \hspace{.5cm}
   \includegraphics[width = 0.4 \textwidth]{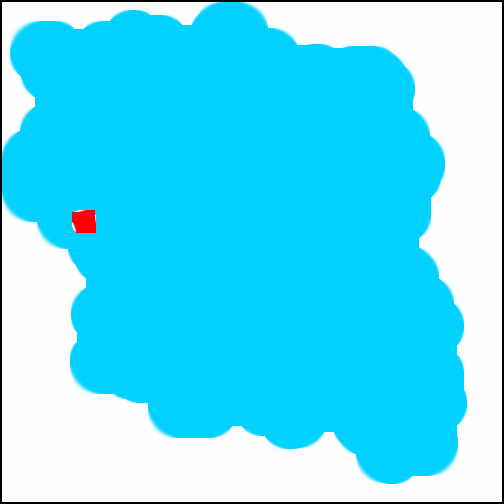}\\
   \vspace{.75cm}
 \includegraphics[width = 0.4 \textwidth]{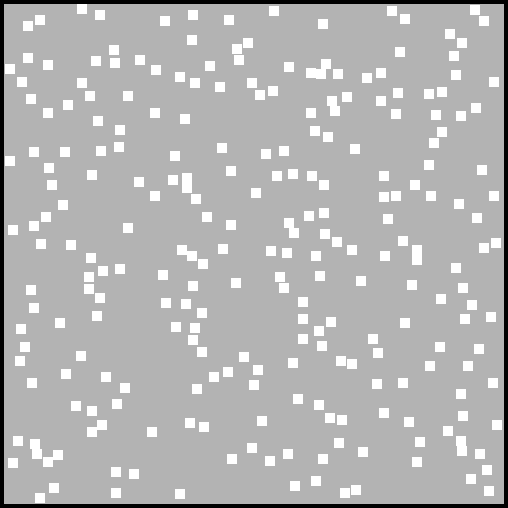} \hspace{.5cm}
  \includegraphics[width = 0.4 \textwidth]{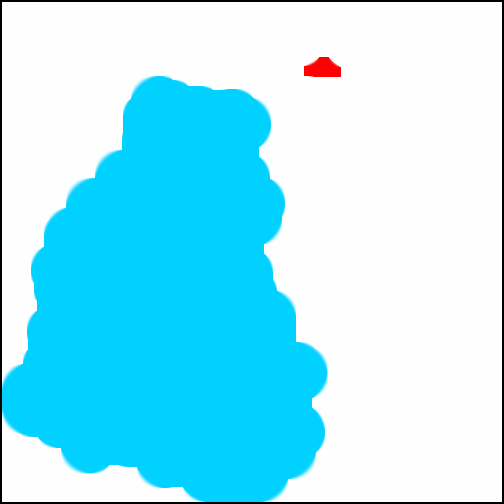}
 \end{center}
\caption{Two random distributions of the same number of particles in a square. On the left lies the distribution of particles and on the right the computed minimizer. Note that the values are $\{-2.51, 712\}$ (up) and $\{-4.21, 702\}$ (down) while $Y_c^{\text{up}} = 7.89\cdot 10^{-3}$ and $Y_c^{\text{low}} = 6.75\cdot 10^{-3}$. Here again, the side length of the domain is $1$ and the color of the negative values has been rescaled.}
\label{fig:randpart}
\end{figure}

\section*{Acknowledgments}
This research was supported by the Austrian Science Fund (FWF) through the National Research Network `Geometry+Simulation' (NFN S11704). We would like to thank Ian Frigaard (UBC) for useful discussions.

\bibliographystyle{plain}
\bibliography{IglMerSch18-final}
\end{document}